\documentclass[12pt,twoside]{amsart}
\input{includeNice3}
\usepackage{mabliautoref}
\usepackage{amssymb}
\usepackage{amscd}
\usepackage[abbrev,alphabetic]{amsrefs}
\usepackage{csquotes}
\usepackage{hyperref}
\usepackage[a4paper,margin=1in]{geometry}  
\usepackage{CJKutf8}
\usepackage{soul}

\usepackage{tikz-cd}
\usepackage{caption}

\usepackage{url}
\RequirePackage{booktabs, multirow}
\RequirePackage{pgf}
\newcommand{\rar}{\rightarrow}	
\newcommand{\drar}{\dashrightarrow}

\newcommand{\ZZ}{\mathbb{Z}}
\newcommand{\QQ}{\mathbb{Q}}
\newcommand{\RR}{\mathbb{R}}
\newcommand{\CC}{\mathbb{C}}

\def\O#1.{\mathcal {O}_{#1}}
\def\K#1.{K_{#1}}
\def\bM#1.{\mathbf{M}_{#1}}
\def\bN#1.{\mathbf{N}_{#1}}
\def\subs#1.{_{#1}}					
\def\sups#1.{^{#1}}

\theoremstyle{plain}

\title[]
{On the boundedness of elliptic Calabi--Yau 4-folds}

\author[S. Filipazzi]{Stefano Filipazzi} 
\author[F. Xu]{Fulin Xu}

\subjclass[2020]{Primary: 14E30, 14J27, 14J32;
Secondary: 14D06.}

\keywords{Calabi--Yau 4-folds, elliptic fibrations, boundedness, minimal model program, Morrison--Kawamata cone conjecture.}

\thanks{SF was partially supported by ERC starting grant  \#804334 and subsequently by
Duke University.}

\address{Department of Mathematics, Duke University,
120 Science Drive,
117 Physics Building,
Campus Box 90320,
Durham, NC 27708-0320,
USA}
\email{stefano.filipazzi@duke.edu}

\address{Qiuzhen College, Tsinghua university,  
Haidian district, 100084, Beijing, China}
\email{xufl23@mails.tsinghua.edu.cn}

\begin{document}

    \begin{abstract}
        In this work, we settle the boundedness of a vast class of elliptic Calabi--Yau 4-folds.
        In particular, we show that any elliptic Calabi--Yau 4-fold that is not crepant to the quotient of a product $Y \times E$, where $E$ is an elliptic curve and $Y$ a Calabi--Yau 3-fold, belongs to finitely many algebraic families.
        In particular, the statement applies whenever the elliptic fibration of the Calabi--Yau 4-fold is not isotrivial.
        We also provide partial evidence for the boundedness of the middle Betti number of Calabi--Yau 3-folds and study the index of fibered $K$-trivial 4-folds.
    \end{abstract}

\maketitle

\vspace{-0.3cm}

\tableofcontents

\section{Introduction}

Projective varieties with mild singularities and numerically trivial canonical bundle (in short, {\it $K$-trivial varieties}) form one of the fundamental classes of algebraic varieties.
In dimensions 1 and 2, $K$-trivial varieties are fully understood.
In dimension 1, they coincide with elliptic curves.
In dimension 2, smooth $K$-trivial varieties are completely characterized by the Enriques--Kodaira classification
\cite{Enriques,Kod1,Kod2,Kod3,Kod4}, 
and they fall into four possible cases:
K3 surfaces, abelian surfaces, Enriques surfaces, and bielliptic surfaces.
In particular, K3 surfaces and abelian surfaces belong to finitely many deformation families of (not necessarily algebraic) compelx surfaces, while Enriques and bielliptic surfaces belong to finitely many algebraic families.
Then, $K$-trivial surfaces with du Val singularities can be studied through their minimal resolution, which is a smooth $K$-trivial surface.
Singular $K$-trivial surfaces with worse than du Val singularities belong to finitely many algebraic families due to work of Alexeev \cite{Ale94}.

In higher dimensions, very little is known about the possible deformation classes of $K$-trivial varieties.
To this end, it is helpful to focus first on some distinguished classes of $K$-trivial varieties.
Indeed, due to the {\it Beauville--Bogomolov decomposition}, any $K$-trivial variety admits a quasi-\'etale cover that splits as product of {\it abelian varieties}, {\it Calabi--Yau varieties}, and {\it irreducible symplectic varieties};
see \cite{Bog74, Beau83, GKP16, Dru18, GGK19, HP19}.
In any dimension, abelian varieties belong to a unique complex deformation family, the one of complex tori.
Irreducible symplectic varieties are the higher dimensional counterpart of (singular) K3 surfaces.
While a general classification of irreducible symplectic varieties is at the moment out of reach,
the recent work \cite{EFGMS}
together with the hyperk\"ahler SYZ conjecture--see 
\cite{MR2585581}
and references therein--provide evidence towards the expectation that irreducible symplectic varieties belong to finitely many complex deformation families in each dimension.
On the other hand, Calabi--Yau varieties remain elusive already in dimension 3.

A fundamental question in 3-fold geometry, originally due to Reid and Yau, asks whether smooth Calabi--Yau 3-folds have finitely many topological types;
see, e.g., \cite{MR909231,MR2537089}.
In algebraic geometry, a natural approach to this question is to prove the stronger statement that smooth Calabi--Yau 3-folds form finitely many algebraic families.
While this question remains open, in recent years there has been a lot of progress in understanding special classes of Calabi--Yau varieties.
In \cite{Gro94}, Gross showed that elliptic Calabi--Yau 3-folds are {\it birationally bounded}.
Said otherwise, there are finitely many algebraic families such that every elliptic Calabi--Yau 3-fold is birational to a fiber of one of said families.
The study of elliptically fibered Calabi--Yau varieties was then revived by works of Di Cerbo and Svaldi, and Birkar, Di Cerbo, and Svaldi \cite{DCS21,BDCS24}, where the authors focused on the case of fibrations admitting a rational section.
As in the work of Gross, these results settle the birational boundedness of the class under consideration.
Lastly, Englel, Filipazzi, Greer, Mauri, and Svaldi settled the birational boundedness of Calabi--Yau varieties fibered in abelian varieties in any dimension \cite{EFGMS}.

It is known that a Calabi--Yau variety may have infinitely many marked minimal models and at least finitely many distinct unmarked minimal models;
see, e.g., \cite{fryers2001movablefanhorrocksmumfordquintic}.
In particular,
two non-isomorphic Calabi--Yau varieties can be birational.
Thus,
while the above results succeed in classifying the birational class of fibered Calabi--Yau varieties, their methods fail in distinguishing isomorphism classes.
In particular, the methods from the MMP are not well suited to distinguish two birational Calabi--Yau varieties, as they are connected by a sequence of flops \cite{Kaw08}.

The Morrison--Kawamata cone conjecture provides a clear expectation for the structure of the nef and movable cones of a Calabi--Yau variety \cite{Mor93,Mor96,Kaw97}.
In particular, it predicts that every Calabi--Yau variety is birational to only finitely many other Calabi--Yau varieties, up to isomorphism.
The cone conjecture is fully settled only in dimension 2 \cite{Sterk,Namikawa,Kaw97,Tot10}, while only special cases are known in higher dimensions;
see, e.g., \cite{Mar11,PS12,Ogu14,LP13,PS23,FXu24,MQ24,GLSW24,Gac25}.

The perspective provided by the cone conjecture can be used to enhance statements regarding the birational boundedness of Calabi--Yau varieties to full boundedness statements;
see, e.g., \cite[\S~5]{FS20}.
To this end, the relative version of the conjecture, first intruduced by Kawamata in \cite{Kaw97}, plays a central role.
In \cite{FHS21}, Filipazzi, Hacon, and Svaldi first proved this perspective fruitful by promoting the results of Gross in \cite{Gro94} to a boundedness statement for elliptic Calabi--Yau 3-folds.
Further recent developments studying the cone conjecture in a relative setting include \cite{LZ25,Li23,MS24,Lutz24,CLLZ25}.

This work follows the philosophy of \cite{FHS21} and settles the boundedness of a vast class of elliptic Calabi--Yau 4-folds.
More precisely, we prove the following statement.

\begin{theorem}[cf.~\autoref{main_thm_technical}]\label{intro_thm_CY}
    Let $\mathfrak F _{CY,ell,\sigma}^4$ denote the set of elliptic Calabi--Yau 4-folds not of product type.\footnote{We say that an elliptic Calabi--Yau variety $X$ with elliptic fibration $X \to Y$ is of {\it product type} if $X$ is crepant birational to a quotient $(Z \times F)/G$, where $F$ is an elliptic curve and $G$ is a finite group acting componentwise, and the fibration $X \to  Y$ is birationally equivalent to $(Z \times F)/G \to Z/G$.
    In particular, \autoref{intro_thm_CY} applies whenever the elliptic fibration is not isotrivial.}
    Then, the set $\mathfrak F _{CY,ell,\sigma}^4$ is bounded.
    That is, elliptic Calabi--Yau 4-folds not of product type appear as fibers of finitely many algebraic families.
\end{theorem}

\autoref{intro_thm_CY} immediately yields the finiteness of the topological
types of many elliptic Calabi--Yau 4-folds, which has relevant consequences in string theory;
see, e.g., \cite{Martucci:2022krl}.

\subsection{Additional results}

\subsubsection{Hodge bundle for families of CY 3-folds}
To a family $f \colon X \to Y$ of $K$-trivial varieties, one can associate a $\mathbb Q$-line bundle $\bM Y.$ on $Y$, called {\it Hodge bundle}.
Classically, this object is defined via Hodge-theoretic methods--see, e.g., \cite{Kol07}--but it can also be defined using methods of birational geometry as introduced in \cite{FM00,PS09}.
By \cite{BFMT}, up to taking a suitable resolution of $Y$, $\bM Y.$ is semiample and its Kodaira dimension measures the variation of the fibers of $f$;
see also \cite{Fuj03,Amb04,Amb05}.

If $f$ is an elliptic fibration, work of Kodaira and Fujita shows that $\mathcal{O}_Y(12 \mathbf M _Y) \simeq j^*\mathcal{O}_{\mathbb P ^1}(1)$, where $Y$ is a model where the $j$-map $j \colon Y \dashrightarrow \mathbb P ^1$ is a morphism;
see \cite{Kod1,Kod2,Fuj86}.
In general, Prokhorov and Shokurov predicted that there is an integer $C>0$, only depending on $\dim X - \dim Y$, such that $|C \bM Y.|$ is a basepoint free linear series.
Then, the prediction also extends to fibrations in log Calabi--Yau pairs, where the boundary coefficients also play a role;
see \cite[Conjecture 7.13.3]{PS09}.
At the moment, this conjecture is known in relative dimension up to 2 by combining \cite{Fuj03,PS09,Fil20,babwild,EFGMS}.

In general, for a fixed relative dimension $n=\dim X - \dim Y$, already the weaker property of finding a uniform $C'>0$ such that $C' \bM Y.$ is integral is open.
By \cite[\S~3]{FM00}, the existence of such a constant $C'$ is implied by a bound on the middle Betti numbers $K$-trivial $n$-folds.
This is a very hard question that is wide open already in dimension 3.

As replacement of a bound on the size of the Betti numbers, we consider
another important invariant of a $K$-trivial variety $X$:
its {\it index $i(X)$}, defined as the smallest positive integer $I$ such that $I \cdot K_X \sim 0$.
It is expected that, in every dimension, the index of a $K$-trivial variety--or, more generally, of a log Calabi--Yau pair with fixed coefficients set--can attain only finitely many values;
see, e.g., \cite{Ale94,MR2075967,YXu}.
The index conjecture is settled in dimension up to 3 by contributions of several mathematicians;
see \cite{Kaw86,Mor86,Ale94,PS09,Jia21,JL21,YXu}.

In this work, we utilize the index conjecture for semi-log canonical 3-folds, settled in \cite{YXu}, as replacement of a bound on the Betti numbers of Calabi--Yau 3-folds.
Then, combining it with methods developed by Birkar in \cite{Bir19}, we can show the following theorem:

\begin{theorem}[{cf. \autoref{thm_cartier_index}}]\label{intro:thm_cartier_index}
    Let $f \colon X \to Y$ be a fibration such that the general fiber is a 
    log canonical
    $K$-trivial 3-fold, and let $\bM.$ denote the Hodge bundle of the fibration (as b-divisor).
    Then,  there is a positive integer $I > 0$ such that $I \cdot \mathbf M$ is b-Cartier (in particular, integral on every birational model of $Y$).
\end{theorem}

While \autoref{intro:thm_cartier_index} bypasses the use of an estimate on the middle Betti number of Calabi--Yau 3-folds, it can be considered as a first general evidence that Betti numbers of Calabi--Yau 3-folds shall indeed be bounded.

We refer to \autoref{thm_cartier_index} for a complete version of \autoref{intro:thm_cartier_index} and to \autoref{thm:effective_freeness} for some cases of the effective freeness conjecture of Prokhorov--Shokurov.

\subsubsection{Index of fibered $K$-trivial 4-folds}
As immediate application of \autoref{intro:thm_cartier_index}, we study some cases of the index conjecture for $K$-trivial varieties in dimension 4.
In dimension 4, the conjecture is known for log canonical pairs that are not canonical by \cite{JL21,Jiao24} and for smooth 4-folds by \cite{MR4903129}.
In particular, in dimension 4, only the case of non-smooth canonical $K$-trivial varieties is open.

In this work, we show that the presence of a fibration can be exploited to study the index of a $K$-trivial variety, even when it is not possible to settle stronger statements such as boundedness.
In particular, we settle the following special case of the index conjecture in dimension 4.

\begin{theorem}[{cf. \autoref{index_thm_pairs}}]\label{intro_thm_index}
    There exists a positive integer $C > 0$ such that the following holds.
    Let $X$ be a projective log canonical
    4-fold with $K_X \equiv 0$.
    Further assume that $X$ admits a fibration $f  \colon X \to Y$ with $\dim Y > 0$.
    Then, the index $i(X)$ divides $C$, i.e., $C \cdot K_X \sim 0$ holds.
\end{theorem}

\subsubsection{Iitaka volume of log canonical 4-folds}
A second immediate application of \autoref{intro:thm_cartier_index} pertains to the study of the Iitaka volume of log canonical 4-fold pairs $(X,\Delta)$.
If $(X,\Delta)$ is a projective log canonical pair such that $h^0(X,\mathcal{O}_X(\lfloor m(K_X+\Delta)\rfloor))>0$ for some $m \geq 1$, then the rate of growth in $m$ (for $m$ sufficiently divisible) is asymptotically a polynomial of degree at most $\dim X$.
Said degree is the Kodaira dimension of the pair $(X,\Delta)$.
Thus, $\kappa$ is the least non-negative integer satisfying $h^0(X,\mathcal{O}_X(m(K_X+\Delta))) = O( m^\kappa)$ and $0 \leq \kappa \leq \dim X$.
Then, the Iitaka volume of $(X,\Delta)$ is defined as $\mathrm{vol}_\kappa (X,\Delta) = \limsup \frac{h^0(X,\mathcal{O}_X(\lfloor m(K_X+\Delta)\rfloor))}{m^\kappa / \kappa !}$.
In particular, together with the Kodaira dimension, the Iitaka volume is an invariant that measures the rate of growth of the pluricanonical sections of $(X,\Delta)$.

If $\kappa = \dim X$, one retrieves the usual notion of volume of $(X,\Delta)$.
Work of Hacon--McKernan--Xu showed that the arithmetic properties of the possible volumes $\mathrm{vol}_{\dim X}(X,\Delta)$ are crucial in the study of the birational geometry of varieties of general type;
see \cite{HMX13,HMX18}.
In particular, extending work of Alexeev in dimension 2, they showed that, under natural assumptions, the possible values form a set satisfying the descending chain condition (DCC);
see \cite{Ale94,HMX18}.

Motivated by this result, Li introduced an analogous conjecture for the Iitaka volume of pairs of intermediate Kodaira dimension \cite[Conjecture~1.6]{Li24}.
This conjecture is known up to dimension 3 by work of Chen--Han--Li \cite{CHL25};
see also \cite{Bir21} for related results.

\autoref{intro:thm_cartier_index} allows us to prove \cite[Conjecture~1.6]{Li24} conjecture in dimension 4.
In particular, we obtain the following:
\begin{theorem}[{cf. \autoref{iitaka volume d-3} and \autoref{cor:iitaka_volume}}]
    Let $\Gamma \subset [0,1]$ be a DCC set of rational numbers, and let $d$ and $\kappa$ be non-negative integers satisfying $\kappa \leq d$.
    Consider
    $$
    \mathfrak{P}^{\Gamma}_{\mathrm{lc}}(d,\kappa) = \{(X,\Delta)|(X,\Delta) \text{ \rm is a projective lc pair},\dim X = d, \kappa(X,\Delta)=\kappa,\mathrm{coeff(\Delta)\subset \Gamma} \}
    $$
    and
    $$
    \mathfrak{V}^\Gamma_{\textrm{lc}}(d,\kappa) = \{\mathrm{vol}_\kappa(X,\Delta)|(X,\Delta) \in \mathfrak{P}^{\Gamma}_{\mathrm{lc}}(d,\kappa) \}.
    $$
    Then, $\mathfrak{V}^\Gamma_{\textrm{lc}}(d,\kappa)$ is a DCC set for all $\kappa \geq d-3$.
    In particular, $\mathfrak{V}^\Gamma_{\textrm{lc}}(d,\kappa)$ is a DCC set for $d \leq 4$.
\end{theorem}

\subsubsection{Some CY varieties of product type and BB decomposition in family}
\autoref{intro_thm_CY} does not apply to Calabi--Yau 4-folds of product type in general, because the Morrison--Kawamata cone conjecture is not known for $K$-trivial 3-folds in full generality.
On the other hand, in \cite{FXu25}, some special cases of $K$-trivial 3-folds are treated.
In particular, if $Z$ is a $K$-trivial 3-fold with positive {\it augmented irregularity} $\widetilde{q}(Z)>0$, then the cone conjecture is known.
Said otherwise, $Z$ is a 3-fold with
klt singularities, $K_Z \sim_{\mathbb Q} 0$, and there is a quasi-\'etale cover $Z' \to Z$ with $h^1(Z',\mathcal{O}_{Z'}) > 0$.
This is equivalent to excluding the possibility that $Z$ admits a quasi-\'etale cover that is a Calabi--Yau 3-fold;
see \autoref{rmk_cy}.
Thus, we can improve upon \autoref{intro_thm_CY} and obtain boundedness for elliptic Calabi--Yau 4-folds that are not crepant to a quotient $(Y \times F)/G$, where $Y$ is a Calabi--Yau 3-fold.
This is the content of \autoref{main_thm_technical}.\autoref{main_thm_technical_part_b}.

In order to establish \autoref{main_thm_technical}.\autoref{main_thm_technical_part_b}, we need to understand how the augmented irregularity behaves in families of $K$-trivial varieties.
This is accomplished in \autoref{section_bb}.
In particular, we show that, given a family $\mathcal{X} \to T$ of $K$-trivial varieties, up to an \'etale locally closed decomposition of $T$, a Beauville--Bogomolov decomposition of a fiber $\mathcal{X}_t$ is the specialization of a Beauville--Bogomolov decomposition of the generic fiber;
see \autoref{bb_in_family}.
This is particularly useful in studying boundedness questions, where one is allowed to replace the original parameter space $T$ with a locally closed decomposition and further covers.

\subsection{Proof strategy}
The first input towards \autoref{intro_thm_CY} is the recent birational boundedness result in \cite{EFGMS} for fibered Calabi--Yau varieties.
In particular, by \cite[Theorem~A]{EFGMS}, the varieties in \autoref{intro_thm_CY} are birationally bounded.
Then, to obtain boundedness, we wish to utilize a suitable version of the Morrison--Kawamata cone conjecture.
Building upon work of Kawamata \cite{Kaw97}, in \cite{FHS21}, the authors showed that the cone conjecture holds for elliptic fibrations in any dimensions.
In particular, if $X \to Z$ is an elliptic Calabi--Yau 4-fold, by \cite[Theorem~6.18]{FHS21}, it suffices to study the cone conjecture for the natural log Calabi--Yau pair induced on the 3-fold $Z$ by the canonical bundle formula, rather than the full cone conjecture for the more complicated 4-fold $X$.
In particular, we reduce to considering the cone conjecture for log Calabi--Yau 3-fold pairs $(Z,\Delta)$, where $Z$ is rationally connected;
see \autoref{prop:rc}.

Then, the recent progress towards the Morrison--Kawamata cone conjecture in low dimension by Xu
\cite{FXu24,FXu25} naturally applies in studying the birational models of the 3-folds $(Z,\Delta)$.
In particular, unless $Z$ is the quotient of a Calabi--Yau 3-fold, we can apply the results of \cite{FXu24,FXu25} to deduce the finiteness of the birational models of $(Z,\Delta)$.
Yet, in order to obtain a boundedness statement and argue along the lines of \cite[Theorem~6.18]{FHS21}, we need a relative version of the Morrison--Kawamata cone conjecture to be applied to the entire family of 3-folds, rather than to one 3-fold at the time.
To this end, we can apply recent work of Li, who provided necessary criteria to obtain a relative version of the cone conjecture from the corresponding global version;
see \cite{Li23}.

In particular, we obtain the following, which constitutes the technical core towards \autoref{intro_thm_CY}.
\begin{theorem}[cf.~\autoref{log boundedness Iitaka dim > 0}]\label{intro_thm_bases}
    Let $\mathfrak B ^4_{CY, ell}$ denote the set of 3-fold pairs $(Z,\Delta)$ appearing as bases of an elliptic fibration $\pi \colon X \to Z$, where $X$ is a 
    Calabi--Yau 4-fold not of product type.
    Then, the set $\mathfrak B ^4_{CY, ell}$ is log bounded.
\end{theorem}

If instead $X \to Z$ is an elliptic Calabi--Yau 4-fold of product type, $Z$ may be birational to the quotient of a Calabi--Yau 3-fold, and in this case the cone conjecture is not known.
In this case, the fibration $X \to Z$ is birational to an {\it orbibundle}, as introduced in \cite{Kol15}.
The possible group actions and Calabi--Yau indices for $Z$ are discussed in detail in \autoref{subsection_orbibundle}.

\subsection*{Acknowledgements}
The authors would like to thank Ben Bakker, C\'ecile Gachet, Xintong Jiang, Junpeng Jiao, and Mirko Mauri for helpful discussions.
The second named author would like to thank Professor Birkar for his constant support and guidance. 

\section{Preliminaries}

Throughout this paper, we work over the field of complex numbers $\mathbb C$.
We refer to \cite{KM98} for the standard terminology in birational geometry and to \cite{kollar_moduli} for the notions of family of pairs.

A \emph{contraction} is a projective morphism 
$f\colon X \rar Z$ 
of quasi-projective varieties with 
$f_\ast  \O X. = \O Z.$. 
If $X$ is normal, then so is $Z$ and the fibers of $f$ are connected.
A contraction $f \colon X \to Z$ with $\dim Z < \dim X$ is called \emph{fibration}.

By {\it irregularity $q(X)$} of a variety $X$ we mean $h^1(X,\mathcal{O}_X)$.
By {\it augmented irregularity $\tilde q (X)$} of a variety $X$ we mean the maximum among the irregularities of quasi-\'etale covers of $X$.
If $X$ is a projective variety of klt type, 
it has rational singularities, and thus we have $q(X)=q(X')$ for any resolution of singularities $X' \to X$. 
Under the same assumptions, we further have $h^1(X,\mathcal{O}_X)= h^0(X,\Omega_X^{[1]})$ by \cite[Theorem~1]{Sch16}.

\begin{remark}
    Throughout the paper, we work over $\mathbb C$ and some proofs rely on topological arguments.
    Nevertheless, the main results regarding boundedness of varieties or of their discrete invariants remain valid over any algebraically closed field of characteristic 0 by the Lefschetz principle.
\end{remark}

\subsection{B-divisors}\label{b-div.subs}
Let $\mathbb{K}$ denote $\ZZ$, $\QQ$, or $\RR$.
Given a normal variety $X$, a {\it $\mathbb{K}$-b-divisor} $\mathbf{D}$ is a (possibly infinite) sum of geometric valuations $V_i$ of $\CC (X)$ with coefficients in $\mathbb{K}$,
\begin{align*}
\mathbf{D}= \sum_{i \in I} b_i V_i, \; b_i \in \mathbb{K},
\end{align*}
such that, given any normal variety $X'$ birational to $X$, only finitely many valuations $V_{i}$ have a center of codimension 1 on $X'$.
The {\it trace} $\mathbf{D}_{X'}$ of $\mathbf{D}$ on $X'$ is the $\mathbb{K}$-Weil divisor
\begin{align*}
\mathbf{D}_{X'} 
\coloneqq \sum b_i D_i,
\end{align*}
where the sum runs over those valuations $V_{i}$ that have divisorial center on $X'$, denoted by $D_{i}$ in the formula.

Given a $\mathbb{K}$-b-divisor $\mathbf{D}$ on $X$, we say that $\mathbf{D}$ is a {\it $\mathbb{K}$-b-Cartier} $\mathbb{K}$-b-divisor if there exists a birational model $X'$ of $X$ such that $\mathbf{D}_{X'}$ is $\mathbb K$-Cartier on $X'$ and for any model 
$r \colon X''  \rar X'$, $\mathbf{D}_{X''} = r^\ast \mathbf{D}_{X'}$ holds.
When this is the case, we say that $\mathbf{D}$ descends to $X'$ and we write $\mathbf{D}= \overline{\mathbf{D}_{X'}}$.
We say that $\mathbf{D}$ is {\it b-effective}, if $\mathbf{D}_{X'}$ is effective for any model $X'$.
We say that $\mathbf{D}$ is {\it b-nef} (resp. {\it b-semiample}), if it is $\mathbb{K}$-b-Cartier and, moreover, there exists a birational model $X'$ of $X$ such that $\mathbf{D}= \overline{\mathbf{D}_{X'}}$ and $\mathbf{D}_{X'}$ is nef (resp. semiample) on $X'$.
These notions can be extended analogously to the relative case.

In all of the above, if $\mathbb{K}= \ZZ$, we will systematically drop it from the notation.

\subsection{Coefficients of divisors}

Given a divisor $D$ on a variety $X$, we write $\mathrm{coeff}(D)$ for the {\it set of coefficients of $D$}.
Said otherwise, if $D=\sum_{i \in I}a_i P_i$ is a finite sum where each $P_i$ is a prime divisor and $P_i \neq P_j$ for $i \neq j$, then $\mathrm{coeff}(D)=\{a_i\}_{i \in I}$.
If we have a divisor $D$ and a set $\Gamma \subset \mathbb R$, we say that {\it the coefficients of $D$ belong to (or are in) $\Gamma$} if $\mathrm{coeff}(D) \subset \Gamma$.

A set $\Gamma \subset \mathbb R$ satisfies the {\it descending chain condition (DCC for short)} if there is no infinite strictly decreasing sequence with values in $\Gamma$.
Similarly, we say $\Gamma$ satisfies the {\it ascending chain condition (ACC for short)} if there is no infinite strictly increasing sequence with values in $\Gamma$.

The {\it standard set} is defined as $\Phi = \{1-\frac{1}{m}|m \in \mathbb N_{\geq 1} \} \cup \{1\}$ and it is a DCC set.

\subsection{Generalized pairs and singularities}
\label{sect.gen.pairs.sings}

We recall the definition of generalized pairs, first introduced in~\cite{BZ16}.
This is a generalization of the classic notion of pairs.

\begin{definition}
A {\em generalized sub-pair} $(X,B, \mathbf{M})/Z$ over $Z$ is the datum of:
\begin{enumerate}
\item a projective morphisms $X  \rar Z$ of normal quasi-projecive varieties;
\item an $\mathbb R$-Weil divisor $B$ on $X$;
\item an $\mathbb R$-b-Cartier b-divisor $\mathbf{M}$ over $X$ which descends to a Cartier divisor $\mathbf{M}_{X'}$ on some birational model $X' \rightarrow X$, and $\mathbf{M}_{X'}$ is relatively nef over $Z$.
\end{enumerate}
We require that $K_X +B+ \mathbf{M}_X$ is $\mathbb R$-Cartier.
When $B$ is effective, we say that $(X,B,\bM.)/Z$ is a generalized pair.
If $\mathbf M \neq 0$ and $Z= \mathrm{Spec}(\CC)$, we drop $Z$ from the notation.
Similarly, if $\bM. = 0$, we drop the word generalized, as this retrieves the usual notions of sub-pair and pair;
in this case, we also drop $Z$ from the notation, regardless of $Z$ being $\mathrm{Spec}(\CC)$.
\end{definition}

The usual notions of singularities of pairs in birational geometry naturally extend to the setup of generalized pairs; see \cite{BZ16}.

A variety $X$ is said to be of {\it klt-type} if there exists a boundary divisor $B$ such that $(X,B)$ is a klt pair.

\subsection{Crepant pairs and isomorphisms in codimension 1}
Let $(X,\Delta)$ and $(X',\Delta')$ be two pairs.
We say that they are \emph{crepant birational} to one another if there exist projective birational morphisms $p \colon Y \to X$ and $q \colon Y \to X'$ from a normal variety $Y$ such that
$p^*(K_X +\Delta)=q^*(K_{X'} +\Delta')$.

We say that two varieties $X$ and $X'$ are \emph{isomorphic in codimension 1} if there exist big open subsets $U \subset X$ and $U' \subset X'$ and an isomorphism $U \cong U'$.

\subsection{Log Calabi--Yau pairs and all that}
\label{subsection_log_cy}

In this section, we recall the notion of log Calabi--Yau pair and we discuss properties of fibrations whose fibers are log Calabi--Yau pairs.

\begin{definition}\label{def_k_trivial}
    Let $X \to S$ be a contraction of normal quasi-projective varieties, and let $(X,\Delta)$ be a pair.
    To denote the presence of the contraction, we write $(X/S,\Delta)$.
    We say that $(X/S,\Delta)$ is a \emph{log Calabi--Yau pair over $S$} if $(X,\Delta)$ is log canonical, and $K_X + \Delta \sim _{\RR,S} 0$ holds.
    If $\Delta = 0$, we say that $X$ is a \emph{$K$-trivial variety over $S$}.
    Lastly, if $S=\mathrm{Spec}(\mathbb C)$, we drop $S$ from the notation and we omit ``over $S$'' in the above terminology.
\end{definition}

When $X$ is a $K$-trivial variety with klt singularities, it can be decomposed, up to a \emph{quasi-\'etale cover} (i.e., a finite cover that is \'etale over the regular locus), into a product of Calabi--Yau, Abelian, and irreducible symplectic varieties; see \cite[Theorem 1.5]{HP19}.
Among other properties, these special classes of varieties are identified through the behavior of their sheaves of reflexive differentials.
We recall that, given a normal variety $X$, the sheaf of \emph{$i$-th reflexive differentials}, denoted by $\Omega_X^{[i]}$, is defined by $\iota _* \Omega _{X_{\rm reg}}^i$, where $\iota \colon X_{\rm reg} \to X$ denotes the natural inclusion of the regular locus of $X$.

In this work, we are mainly interested in Calabi--Yau varieties.
Thus, for the reader's convenience, we recall the following definition.

\begin{definition}\label{def_cy}
A projective variety $X$ with canonical singularities is called \emph{Calabi--Yau variety} if the following properties hold:
\begin{enumerate}
    \item $K_X \sim 0$; and
    \item\label{condition_covers} for every quasi-\'etale cover $Y \to X$ and integer $0 < i < \dim X$, $H^0(Y,\Omega_Y^{[i]}) = 0$ holds.
\end{enumerate}
Furthermore, if $X$ is smooth, we say it is a \emph{Calabi--Yau manifold}.
\end{definition}

\begin{remark}
If $X$ is a Calabi--Yau manifold with $\dim X \geq 2$, then $\pi_1(X)$ is finite; see
\cite{Bea83}.
Thus, its universal cover is itself a Calabi--Yau manifold.
For this reason, we define a \emph{strict Calabi--Yau manifold} to be a simply connected Calabi--Yau manifold.

If $X$ is singular, quasi-\'etale covers are governed by $\pi_1(X_{\rm reg})$.
It is expected that $\pi_1(X_{\rm reg})$ is finite if $X$ is a Calabi--Yau variety with $\dim X \geq 2$, but this is not known.
This is the main reason why condition \autoref{condition_covers} in \autoref{def_cy} has to be checked for all possible quasi-\'etale covers $Y \to X$.

In this paper, a {\it Calabi--Yau $n$-fold} denotes a Calabi-Yau variety of dimension $n$.
\end{remark}

\begin{remark}
    Being a Calabi--Yau variety is not a property of a crepant birational equivalence class.
    For instance, a singular Kummer surface is not a Calabi--Yau variety, since it has a quasi-\'etale cover that is an abelian surface.
    On the other hand, its minimal resolution is a smooth K3 surface, which is a Calabi--Yau manifold.
    \footnote{We remark that K3 surfaces are also hyperk\"ahler varieties and, for this reason, some authors require Calabi--Yau varieties to have dimension at least 3.
    Nevertheless, there are examples of Calabi--Yau varieties that are crepant birational to $K$-trivial varieties that are not Calabi--Yau also in dimension 3;
    see, e.g., \cite{Oguiso96,Gac24}.}
\end{remark}

If a $K$-trivial variety (resp. Calabi--Yau variety) $X$ admits a contraction $f \colon X \to Y$ such that the general fiber of $f$ is an elliptic curve, we say that $X$ is an \emph{elliptic $K$-trivial variety} (resp. \emph{elliptic Calabi--Yau variety}).
In this context, we say that $f$ admits a \emph{section} if there exists a morphism $\sigma \colon Y \to X$ such that $f \circ \sigma = id_Y$.
Similarly, we say that $f$ admits a rational section if there exists a rational map $\sigma \colon Y \drar X$ such that $f \circ \sigma =id_Y$.
We often abuse notation and freely use the term (rational) section to refer to the image of $\sigma$.
Lastly, we say that $f$ has a \emph{multisection of degree $d$} if there exists an irreducible subvariety $W \subset X$ of codimension 1 such that $W \to Y$ is generically finite of degree $d$.

\subsection{Lc-trivial fibrations and the canonical bundle formula}
\begin{definition}
\label{lc-trivial.def}
Let $(X, \Delta)$ be a sub-pair with coefficients in $\QQ$.
A fibration $f \colon X \rar Z$ of quasi-projective varieties is an \emph{lc-trivial fibration} if
\begin{enumerate}
    \item $(X,\Delta)$ is sub-log canonical over the generic point of $Z$;
    \item\label{condition_log_disc} $\mathrm{rank} f_\ast  \O X. (\lceil \mathbf{A}^\ast (X,\Delta)\rceil)=1$, where $\mathbf{A}^\ast (X,\Delta)$ is the b-divisor defined in \cite[\S~1.3]{Amb04}; and
    \item there exists a $\QQ$-Cartier $\QQ$-divisor $L_Z$ on $Z$ such that $\K X. + \Delta \sim_\QQ f^\ast  L_Z$.
\end{enumerate}
If furthermore $(X,\Delta)$ is a log canonical pair, we call it \emph{log Calabi--Yau fibration}.
\end{definition}

\begin{remark}
Condition \autoref{condition_log_disc} in \autoref{lc-trivial.def} is automatically satisfied if $\Delta$ is effective over the generic point of $Z$.
\end{remark}

\begin{remark}\label{rmk:cbf}
Whenever we have an lc-trivial fibration $f \colon (X,\Delta) \rar Z$, with $Z$ projective over some base variety $S$, the \emph{canonical bundle formula} naturally induces on $Z$ a structure of generalized pair $(Z,B_Z,\bM.)/S$ with coefficients in $\mathbb Q$.
The divisor $B_Z$ is called \emph{boundary divisor}, while the b-divisor $\bM.$ is called \emph{moduli b-divisor}.
For details, we refer to \cite{Amb04, Amb05, FG14}.
By \cite{BFMT}, $\mathbf M$ is b-semiample.
\end{remark}

\subsection{Generically isotrivial fibrations and orbibundles}
In this section, we recall the notions of generically isotrivial fibration and of orbibundle, as introduced in \cite{Kol15}.

\begin{definition} \label{def_gen_isotrivial}
Let $(X,\Delta)$ be a projective klt pair and let $f \colon (X,\Delta) \to Z$ be an lc-trivial fibration.
We say that $f$ is {\it generically isotrivial} if there exists a non-empty open set $U \subset Z$ over which the pairs $(X_z, \Delta|_{X_z})$, $z \in U$ closed point, are all isomorphic.
\end{definition}

A particular case of generically isotrivial lc-trivial fibration is obtained by taking quotients of products.
This notion was formalized by Koll\'ar in \cite{Kol15}.
We follow the terminology introduced in \emph{loc. cit.}, with the exception that we prefer to use the term ``$K$-trivial'', as introduced in \autoref{def_k_trivial}.

\begin{definition}
    Let $X$ be a $K$-trivial variety with klt singularities, and let $f \colon X \to Y$ be a contraction with $0 < \dim Y < \dim X$.
    We say that $f$ is a \emph{$K$-trivial orbibundle} if it is obtained by the following construction.
    There exist klt $K$-trivial varieties $\tilde Y$ and $F$, a finite group $G$, and faithful representations $\rho_F \colon G \to \mathrm{Aut}(F)$ and $\rho_{\tilde Y} \colon \tilde{Y} \to \mathrm{Aut}(\title{Y})$, such that
    $$
        (f \colon X \to Y) = (\tilde X /G \to \tilde Y /G)
    $$
    holds, where $\tilde X = \tilde Y \times F$ and $G$ acts on $\tilde X$ via $\rho_{\tilde Y} \times \rho_F$.
\end{definition}

The following statement shows that elliptic Calabi--Yau 4-folds that are crepant birational to a $K$-trivial orbibundle can be detected by the Kodaira dimension of the base of the elliptic fibration.

\begin{lemma}\label{lemma:orbibundle_structure}
    Let $X$ be an elliptic $K$-trivial klt 4-fold with elliptic fibration $\pi \colon X \to Y$.
    Let $Y'$ be a small $\mathbb Q$-factorialization of $Y$, and assume $\kappa(-K_{Y'})=0$.
    Then, $X$ is crepant birational to a $K$-trivial orbibundle $(\widehat Y \times F)/G \to \widehat Y /G$.
\end{lemma}

\begin{proof}
    Let $\pi \colon X \to Y$ be a $K$-trivial 4-fold as in the statement.
    By \cite[Proposition 2.9]{Fil24}, up to replacing $X$ with a crepant model that is isomorphic in codimension 1 to the original one, we may assume that $X$ is $\QQ$-factorial and $X \to Y$ factors through $Y'$.
    Thus, we may replace $Y$ with $Y'$ and assume that $Y$ is $\mathbb Q$-factorial.

    Let $(Y,B_Y,\bM.)$ be the generalized pair induced on $Y$ by the canonical bundle formula.
    Then, we have $K_Y+B_Y+\bM  Y. \sim_{\QQ} 0$.
    By assumption, we have $\kappa(-K_Y)=\kappa(B_Y+\bM Y.)=0$.
    By \cite[Theorem 3.3]{Amb05}, we have $\bM Y. \sim_{\mathbb Q} 0$.
    Thus, we have $K_Y+B_Y\sim_{\mathbb Q}0$.

    Fix $0 < \epsilon \ll 1$ and run a $(K_Y+(1+\epsilon)B_Y)$-MMP.
    By \cite{KMM94}, this MMP terminates with a good minimal model $\tilde Y$.
    By our assumptions, $B_Y$ is contracted by the MMP and we have $K_{\tilde Y} \sim_{\mathbb Q} 0$.
    By \cite[Proposition 2.9]{Fil24}, there exists a crepant $\mathbb Q$-factorial model $X'$ that is isomorphic in codimenison 1 to $X$ and admitting a morphism $X' \to \tilde Y$.

    Let $E'$ be the reduced divisor on $X'$ consisting of the divisors whose image in $\tilde Y$ has codimension at least 2.
    Then, for $0 < \epsilon \ll 1$, we run a $(K_{X'} + \epsilon E')$-MMP with scaling relative to $\tilde Y$.
    By \cite{HX13}, the MMP terminates with a relative good minimal model $\tilde X$, and, by \cite[Lemma 3.3]{Bir12}, $X' \dashrightarrow \tilde X$ contracts exactly $E'$.
    Thus, $\tilde X$ has $\mathbb Q$-factorial klt singularities, and every prime divisor on $\tilde X$ dominates $\tilde Y$ or a prime divisor in $\tilde Y$.
    Then, by \cite[Theorem 43]{Kol15}, we conclude that $\tilde X \to \tilde Y$ is isomorphic to an orbibundle.
\end{proof}

\subsection{Fundamental group of log Calabi--Yau pairs}
In this subsection, we collect some statements about the fundamental group of low dimensional $K$-trivial varieties and of the bases of fibered Calabi--Yau varieties.

When $X$ is a strict Calabi--Yau manifold admitting a fibration $X \to Y$, then $Y$ is necessarily rationally connected by \cite[Corollary~5.1]{BDCS24}.
The same strategy of proof shows that the same statement holds if $X$ is a Calabi--Yau variety in the sense of \autoref{def_cy}.

\begin{proposition}[{cf.~\cite[Corollary~5.1]{BDCS24}}]\label{prop:rc}
    Let $X$ be a Calabi--Yau variety admitting a fibration $X \to Y$.
    Then, $Y$ is rationally connected.
    In particular, $Y$ is simply connected.
\end{proposition}

\begin{proof}
    By \autoref{rmk:cbf}, $Y$ is a variety of klt type.
    If $Y$ is not rationally connected,
    then the base $Z$ of its MRC-fibration $Y \drar Z$ is
    a positive-dimesional variety that is not uniruled by \cite[Corollary~1.4]{GHS03}.
    Then, by \cite[Theorem~14]{KL09}, there exist a quasi-\'etale cover $X' \to X$ and positive-dimensional $K$-trivial varieties $F'$, $V'$ such that $X' \simeq F' \times V'$.

    Since $K_X \sim 0$, we have $K_{X'} \sim 0$, and, by the K\"unneth formula, $K_{F'} \sim 0$ and $K_{V'} \sim 0$.
    But then, by pulling back reflexive top forms on $F'$ and $V'$, $X'$ has global reflexive $i$-forms for $i \not \in \{0,\dim X\}$.
    This contradicts \autoref{def_cy}\autoref{condition_covers}.
    In particular, $Y$ has to be rationally connected.
    In turn, by \cite{Tak03}, $\pi_1(Y)=\pi_1(Y')$ holds where $Y' \to Y$ is a resolutions of singularities.
    Then, 
    $Y'$ is smooth and rationally connected, hence simply connected by \cite{KMM92}.
\end{proof}

\begin{proposition}\label{prop:fund_group}
    Let $(X,\Delta)$ be a klt log Calabi--Yau pair with $\dim X \leq 4$.
    Then, $\pi_1(X)$ is virtually abelian, that is, it contains an abelian subgroup of finite index.
    In particular, if the augmented irregularity $\tilde q (X)$ vanishes, then $\pi_1(X)$ is finite.
\end{proposition}

\begin{proof}
    By \cite[Lemma~4.19]{debarre2001higher}, we may replace $X$ with a quasi-\'etale cover.
    Thus, by \cite{MW25}, we may assume that $(X,\Delta)$ splits as a product of a rationally connected variety and a canonical $K$-trivial variety with trivial canonical class.
    As observed in \autoref{prop:rc}, the rationally connected factor is simply connected.
    Thus, we may reduce to considering a canonical $K$-trivial variety $X$ with trivial canonical class and dimension at most 4.
    Then, by \cite[Corollary~8.25]{MR3617779}, $\pi_1(X)$ is almost abelian.
    Thus, the first claim follows.

    Now, further assume $\tilde q (X)=0$.
    Notice that, since $\pi_1(X)$ has an abelian subgroup of finite index, by considering the normal core of said subgroup, $\pi_1(X)$ has an abelian normal subgroup of finite index.
    Thus, we may replace $X$ by the \'etale cover associated to this subgroup.
    In particular, again by \cite[Lemma~4.19]{debarre2001higher}, we may assume that $\pi_1(X)$ is abelian.
    By \cite{Tak03}, $\pi_1(X)$ coincides with the fundamental group of any resolution of singularities.
    In turn, we have $\pi_1(X) = \pi_1(X')$, where $X' \to X$ is a $\mathbb Q$-factorial terminalization.
    Then, by \cite[Claim 5.6.2.3]{Kol93}, $X'$ has a finite \'etale cover $Z' \times A' \to X'$, where $Z'$ is simply connected and $A'$ is an abelian variety.
    In turn, by $\pi_1(X) = \pi_1(X')$, $X$ has a cover with the same property.
    Yet, by $\tilde q (X) = 0$, it follows that $A'=0$ and $\pi_1(X)$ is finite.
\end{proof}

\subsection{Betti numbers of certain low-dimensional $K$-trivial varieties}

\begin{lemma}\label{lemma:betti_number_obibundle}
Let $X$ be a klt $K$-trivial variety of dimension $3$ with positive irregularity $q(X)>0$.
Then, there exists a nonsingular model $W$ of the index $1$ cover of $X$ such that all the Betti numbers of $W$ are bounded by some constant $N$. 
\end{lemma}

\begin{proof}
Replacing $X$ by its index $1$ cover, we may assume $K_X \sim 0$. 

Consider the Albanese morphism $X \to A$.
By assumption, $\dim A > 0$.
By \cite[Theorem 1.1]{Xu20}, there exists an \'etale cover $\pi \colon B\to A$ such that $X\times_A B \simeq F \times B$ for some projective variety $F$.
By the assumptions on $X$, it follows that $F$ is a $K$-trivial variety with canonical singularities.
Moreover, $X \simeq (F \times B)/G$, where $G = \ker \pi$, which acts on $F$ and $B$, and it acts diagonally on $F \times B$. 

When $\dim A = 3$, $X$ is an abelian variety;
see, e.g., \cite[Theorem 1.1]{Xu20}.
In this case, $X$ is nonsingular and all the Betti numbers $b_i(X)\leq \binom{6}{i}\leq 20$ are clearly bounded. 

When $\dim A = 2$, $F$ is an elliptic curve.
In this case, $X$ is nonsingular as smoothness satisfies \'etale descent, and $b_i(X)\leq b_i(F\times B)\leq \binom{6}{i}\leq 20$ since $F\times B$ is an abelian variety. 

When $\dim A = 1$, $F$ is a canonical $K$-trivial surface.
Let $\widetilde F$ be the minimal resolution of $F$.
By uniqueness of the minimal resolution, $G$ acts on $\widetilde F$. Consider $\widetilde X = (\widetilde F \times B)/G$, where $G$ acts diagonally on $\widetilde F \times B$. Since the action of $G$ on $B$ has no fixed points, the quotient map $\widetilde F \times B \to \widetilde X$ is \'etale. Since smoothness satisfies \'etale descent, $\widetilde X$ is a nonsingular model of $X$. Moreover, $b_i(\widetilde X)\leq b_i(\widetilde F \times B) = b_i(\widetilde F)+b_{i-1}(\widetilde F)$.
By the Enriques--Kodaira classification of smooth $K$-trivial surfaces, all the Betti numbers of $\widetilde F$ are bounded. 
\end{proof}

\subsection{Boundedness}
Here, we recall the notion of boundedness for varieties and pairs.
Roughly speaking, this notion formalizes the na\"ive idea that the objects under consideration shall appear as the fibers of finitely many algebraic families.

\begin{definition}
\label{def.bounded.set.pairs}
Let $\mathfrak{D}$ be a set of projective pairs.
\begin{enumerate}
    \item 
We say that $\mathfrak{D}$ is \emph{log bounded} if there exist a pair $(\mathcal{X},\mathcal{B})$ and a projective morphism $\pi \colon \mathcal{X} \rar T$, where $T$ is of finite type, such that for any pair $(X,B) \in \mathfrak{D}$ there exist a closed point $t \in T$ and an isomorphism $f_t \colon \mathcal{X}_t \rar X$ such that $(f_{t})_{\ast} \mathcal{B}_t= B$.
    \item \label{def_bir_bdd}
We say that $\mathfrak{D}$ is \emph{log birationally bounded} if there exist a pair $(\mathcal{X},\mathcal{B})$ and a projective morphism $\pi \colon \mathcal{X} \rar T$, where $T$ is of finite type, such that for any pair $(X,B) \in \mathfrak{D}$ there exist a closed point $t \in T$ and a birational map $f_t \colon \mathcal{X}_t \dashrightarrow X$ such that $\Supp(\mathcal B _t)= \Supp((f_{t}^{-1})_\ast \Supp(B) + E)$, where $E$ is the exceptional divisor of $f_t$.
    \item 
If $\mathfrak D$ is log birationally bounded and for any pair $(X, B) \in \mathfrak D$ the map $f_t$ in \autoref{def_bir_bdd} is an isomorphism in codimension 1, then we say that $\mathfrak{D}$ is \emph{log bounded in codimension 1}.
\end{enumerate}
When a set $\mathfrak D$ of pairs is a set of varieties, i.e., $\Delta=0$ for every $(X, \Delta) \in \mathfrak D$, we simply say that $\mathfrak{D}$ is \emph{bounded}, (resp. \emph{birationally bounded}, \emph{bounded in codimension 1}).
\end{definition}

\subsection{Various cones in the N\'eron--Severi space}

In this section, we recall the notions of various cones in the N\'eron--Severi space. Our starting point is a contraction $f \colon X\to Z$. Throughout this section, when $Z$ is a point, we often drop $f$ (or $Z$) from the notation.

First we recall some standard definitions for different kinds of divisors. 

\begin{definition}
Let $f \colon X\to Z$ be a contraction of normal varieties and $D$ be a Cartier divisor on $X$. 
The divisor $D$ is said to be 
\begin{itemize}
\item[(a)] \emph{$f$-nef} if $(D . C) \geq 0$ holds for any curve $C$ on $X$ which is contracted by $f$;

\item[(b)] \emph{$f$-movable} if $\dim \text{Supp} \text{Coker} (f_*f^*\mathcal{O}_X (D) \to \mathcal{O}_X(D)) \geq 2$;

\item[(c)] \emph{$f$-effective} if $f_*\mathcal{O}_X (D) \neq 0$; and

\item[(d)] \emph{$f$-big} if its restriction to the generic fiber has maximal Kodaira dimension.
\end{itemize}
\end{definition}

The N\'eron--Severi group $$NS(X/Z) = \{\text{Cartier divisor on }X\}/(\text{numerical equivalence over }Z)$$ is known to be finitely generated. 
Let $N^1(X/Z) = NS(X/Z)\otimes_{\mathbb{Z}} \mathbb{R}$ be the N\'eron--Severi space. 
We set $\rho(X/Z) = \dim N^1(X/Z)$ to be the relative Picard number. 
The numerical class of an $\mathbb{R}$-Cartier divisor $D$ in $N^1(X/Z)$ is denoted by $[D]$. 

Now we can define various cones in $N^1(X/Z)$: 

\begin{definition}
Let $f \colon X\to Z$ be a contraction of normal varieties. 
\begin{itemize}
    
\item[(a)] The \emph{$f$-nef cone} is the closed convex cone in $N^1(X/Z)$ generated
by the numerical classes of $f$-nef divisors, denoted by $\bar{\mathcal{A}}(X/Z)$. 

\item[(b)] The \emph{closed $f$-movable cone} is the closed convex cone in $N^1(X/Z)$ obtained as closure of the cone generated
by the numerical classes of $f$-movable divisors, denoted by $\bar{\mathcal{M}}(X/Z)$. 

\item[(c)] The \emph{$f$-pseudo-effective cone} is the closed convex cone in $N^1(X/Z)$ obtained as closure of the cone generated
by the numerical classes of $f$-effective divisors, denoted by $\bar{\mathcal{B}}(X/Z)$. 
\end{itemize}
\end{definition}

We denote the interior $\mathcal{A}(X/Z) \subseteq \bar{\mathcal{A}}(X/Z)$  (resp. $\mathcal{B}(X/Z) \subseteq \bar{\mathcal{B}}(X/Z)$) to be the open convex subcone generated by the numerical classes of $f$-ample divisors ($f$-big divisors) and it is called {\it $f$-ample cone} (resp. \emph{$f$-big cone}). 
Let $\mathcal{M}(X/Z)\subseteq \bar{\mathcal{M}}(X/Z)$ be the subcone generated by $f$-movable divisors. 

We denote the \emph{$f$-effective cone} by $\mathcal{B}^e(X/Z)$:
This is the convex cone generated by $f$-effective Cartier divisors. 
We call $\mathcal{A}^e(X/Z) \coloneqq \bar{\mathcal{A}}(X/Z) \cap \mathcal{B}^e(X/Z)$ and $\mathcal{M}^e(X/Z) \coloneqq \bar{\mathcal{M}}(X/Z) \cap \mathcal{B}^e(X/Z)$ the \emph{$f$-effective $f$-nef cone} and \emph{$f$-effective $f$-movable cone}, respectively. 

To state the cone conjecture, we include some definitions on the correspondence between birational models and various cones. 

\begin{definition}
Let $f \colon  X\to Z$ be a contraction of normal varieties. 
\begin{itemize}
\item[(a)] A \emph{contraction of $X$ over $Z$} is a contraction $g \colon X \to Y$ such that $f$ factors through $g$. 
The corresponding face of the cone $\mathcal{A}^e(X/S)$ is the pullback $g^*\mathcal{A}^e(Y/S)$. 

\item[(b)]  Assume that $X$ is $\mathbb{Q}$-factorial. 
A \emph{marked small $\mathbb{Q}$-factorial modification} (or \emph{marked SQM} for short) of $f \colon  X \to Z$ is a contraction $g \colon X_0 \to Z$ together with a marking $\sigma \colon  X_0 \dashrightarrow X$ such that $f\circ \sigma = g$ and $\sigma$ is an isomorphism in codimension $1$. 
The corresponding chamber in the cone $\mathcal{M}^e(X/S)$ is the pushforward $\sigma_*\mathcal{A}^e(X_0/Z)$.
\end{itemize}
\end{definition}

The following statement allows to compare various cones in the N\'eron--Severi space within a family of varieties.

\begin{lemma}[cf. {\cite[Theorem 4.2]{FHS21}}]\label{deformation lemma}
Let $(X_0,\Delta_0)$ be a projective terminal $\mathbb{Q}$-factorial pair. 
Assume that $K_{X_0} + \Delta_0 \equiv 0$, $H^1(X_0, \mathcal{O}_{X_0}) = 0$ and $H^2(X_0, \mathcal{O}_{X_0}) = 0$. 
Given a deformation $(X,\Delta) \to (T, 0)$ of $(X_0,\Delta_0)$ over a smooth variety $T$ , then $(X,\Delta)$ is terminal
$\mathbb{Q}$-factorial, and $K_X+\Delta \sim_{\mathbb{Q},T} 0$ over a neighborhood of $0 \in T $. 
Furthermore, after an \'etale base change,  the following
facts hold:
\begin{enumerate}
    \item $\bar{\mathcal A}(X_\eta) \supset \bar{\mathcal A}(X_0)$;

    \item $\bar{\mathcal B}(X/T ) \subset \bar{\mathcal B}(X_0)$; and

    \item $\bar{\mathcal M}(X/T ) \supset \bar{\mathcal M}(X_0)$.

\end{enumerate}
\end{lemma}

\begin{proof}
    The proof procedes verbatim as in \cite[proof of Theorem 4.2]{FHS21}.
\end{proof}

\subsection{Morrison--Kawamata cone conjecture}

There are several versions of (weak) cone conjectures, so we fix our notations here.

Given a log Calabi--Yau pair over $Z$ $(X/Z,\Delta)$, we write $\mathrm{Aut}(X/Z,\Delta)$ for the group of automorphisms of $X$ over the identity of $Z$ that maps $\Delta$ to itself.
Similarly, we write $\mathrm{PsAut}(X/Z,\Delta)$ for the group of pseudo-automorphisms of $X$ over the identity of $Z$ that maps $\Delta$ to itself.
We recall that a pseudo-automorphism is a birational self-map that does not contract nor extract any divisor.
In particular, if $(X/Z,\Delta)$ is a terminal log Calabi--Yau pair over $Z$, we have $\mathrm{PsAut}(X/Z,\Delta)=\mathrm{Bir}(X/Z,\Delta)$.

The following is called the {\it weak cone conjecture}, where 
\autoref{weak_nef_conjecture}
is called the {\it weak nef cone conjecture} and
\autoref{weak_movable_conjecture}
is called the {\it weak movable cone conjecture}:

\begin{conjecture}

Let $(X/Z,\Delta)$ be a
klt log Calabi--Yau pair over $Z$, where $X$ is $\mathbb{Q}$-factorial. 
\begin{enumerate}
\item\label{weak_nef_conjecture} The number of $\mathrm{Aut}(X/Z, \Delta)$-equivalence classes of faces of the cone $\mathcal{A}^e(X/Z)$
corresponding to birational contractions or fiber space structures is finite. 

\item\label{weak_movable_conjecture} The number of $\mathrm{PsAut}(X/Z, \Delta)$-equivalence classes of chambers $\mathcal{A}^e
(X_0 /Z)$
in the cone $\mathcal{M}^e
(X/Z)$ corresponding to marked SQMs $X_0 \to Z$ of $X \to Z$ is finite.
\end{enumerate}
\end{conjecture}

The following is called the {\it cone conjecture}, where
\autoref{nef_cone_conj}
is call the {\it nef cone conjecture} and
\autoref{movable_cone_conj}
is called the {\it movable cone conjecture}: 

\begin{conjecture}
Let $(X/Z,\Delta)$ be a klt log Calabi--Yau pair over $Z$, where $X$ is $\mathbb{Q}$-factorial with $\Delta$ an $\mathbb{R}$-boundary.
\begin{enumerate}
    
\item\label{nef_cone_conj} The action of $\mathrm{Aut}(X/Z,\Delta)$ on the effective nef cone $\mathcal{A}^e(X/Z)$ admits a rational polyhedral fundamental domain $\Pi$, in the sense that: 
\begin{itemize}
\item[(1)] $\mathcal{A}^e(X/Z) = \cup_{g\in \mathrm{Aut}(X/Z,\Delta)}g_*\Pi$; and

\item[(2)] $\text{Int }\Pi \cap g_*\text{Int }\Pi= \emptyset$ unless $g_* = 1$. 
\end{itemize}

\item\label{movable_cone_conj} The action of $\mathrm{PsAut}(X/Z,\Delta)$ on the effective movable cone $\mathcal{M}^e(X/Z)$ admits a rational polyhedral fundamental domain $\Pi^\prime$, in the sense that: 
\begin{itemize}
\item[(1)] $\mathcal{M}^e(X/Z) = \cup_{g\in \mathrm{PsAut}(X/Z,\Delta)}g_*\Pi^\prime$; and

\item[(2)] $\text{Int }\Pi^\prime \cap g_*\text{Int }\Pi^\prime= \emptyset$ unless $g_* = 1$. 
\end{itemize}

\end{enumerate}
\end{conjecture}

\begin{remark}
When $X$ is not $\mathbb{Q}$-factorial, we say that the nef (resp. movable) cone conjecture holds for $X$ if the nef (resp. movable) cone conjecture holds for a small $\mathbb{Q}$-factorialization of $X$. 
\end{remark}

It is known that in dimension at most $3$, the movable cone conejcture implies all the other results, see \cite[Theorem 13, Lemma 14, Lemma 15]{FXu24}.

\subsection{Iitaka fibration in families}

The following statements allow to compare the relative Iitaka fibration of a family of pairs with the Iitaka fibration of a special fiber of the family.

\begin{lemma}\label{lemma commutativity of specialization}
Let $(X,B)/S$ be a terminal log Calabi--Yau fibration, $0\in S$ be a smooth point, and $(X_0,B_0)$ be the special fiber.
Assume $(X,B+\sum_i H_i)$ is dlt, where $\{H_i\}$ is the pullback of a system of local parameters of $0\in S$. 
Let $D$ be a relatively nef and big $\QQ$-divisor of $X$ over $S$, such that $D_0$ is nef and big. 
Consider the Iitaka fibration of $X/S$ defined by $D$, say $f \colon X\to X^{amp}$, and the Iitaka fibration of $X_0$ defined by $D_0$, say $f_0 \colon X_0\to X_0^{amp}$. 

Then the following facts hold:

\begin{enumerate}
    \item\label{first item} The natural map $X_0^{amp} \to (X^{amp})_0$ is an isomorphism; and

    \item\label{second item} $f_{0,*}B_0 = (f_*B)_0$ holds under the identification in \autoref{first item}.
\end{enumerate}
\end{lemma}

\begin{proof}
Part \autoref{first item} follows from \cite[Theorem 4.5]{FS25}, while
part \autoref{second item} follows from \autoref{lem:push_fwd}. 
\end{proof}

\begin{lemma}\label{lem:push_fwd}
Let $(X,B)$ be a klt pair, $S = \sum S_i$ be a sum of Cartier divisors such that $(X,S+B)$ is dlt. 
Let $f \colon (X,B+S)\to (Y,B_Y+S_Y)$ be a crepant birational map such that $S_Y$ is Cartier, and $S = f^*S_Y$. 
Let $X_m$ be an irreducible component of the intersection of all irreducible components of $S$ and $Y_m$ be the image of $X_m$ in $Y$, and $f_m \colon X_m \to Y_m$ be the induced map. 

Then, $f_{m*}(B|_{X_m}) = (B_Y)|_{Y_m}$ holds.
\end{lemma}

\begin{proof}
By induction on the number of components of $S$, we may assume $S$ is irreducible. 
Then the restriction coincides with the different by \cite[Proposition 4.5 (4)]{Kol13}, and the different is compatible with crepant birational morphisms by \cite[Proposition 4.6]{Kol13}. 
\end{proof}

\subsection{Quasi-\'etale covers in family}\label{section_bb}

In order to apply the results on the cone conjecture in \cite{FXu25}, it is important to identify the augmented irregularity of the varieties under consideration and to possibly treat this and related invariants in family and compare them to the ones of the generic fiber.
To this end, it is relevant to study quasi-\'etale covers of log Calabi--Yau pairs in family.

For the definition of {\it irreducible symplectic variety}, we refer to \cite{EFGMS}.

\begin{lemma}\label{lemma:Hodge_symmetry}
    Let $X$ be a projective variety of klt-type.
    Then, for every $p \in \mathbb N$, we have $h^p(X,\mathcal{O}_X)=h^0(X,\Omega_X^{[p]})$.
\end{lemma}

\begin{proof}
    Let $\pi \colon X' \to X$ be a projective resolution of singularities.
    Since klt singularities are rational \cite[Proposition 5.13]{KM98}, we have $h^p(X_y,\mathcal{O}_{X}) = h^p(X',\mathcal{O}_{X'})$ for all $p \in \mathbb N$.
    By Hodge symmetry, we then have $h^p(X',\mathcal{O}_{X'}) = h^0(X',\Omega_{X'}^p)$.
    In turn, by \cite[Theorem~1.4]{GKKP11}, we have $h^0(X',\Omega_{X'}^p) = h^0(X,\Omega_{X}^{[p]})$.
\end{proof}

\begin{lemma}\label{lemma:constant_diff_forms}
    Let $f \colon X \to Y$ be a projective flat morphism of complex analytic spaces over a connected base $Y$.
    Further assume that, for every $y \in Y$, $X_y$ is of klt-type.
    Let $p$ be a non-negative integer.
    Then, the dimension $h^0(X_y,\Omega_{X_y}^{[p]})$ is independent of $y \in Y$.
\end{lemma}

\begin{proof}
    Since all fibers have klt singularities, then the dimension $h^p(X_y,\mathcal{O}_{X_y})$ is independent of $y \in Y$ by \cite[Complement 2.62.5]{kollar_moduli}.
    Then, we may conclude by \autoref{lemma:Hodge_symmetry}.
\end{proof}

\begin{proposition}\label{prop:def_CY_IS}
    Let $f \colon X \to Y$ be a proper flat family of $K$-trivial varieties over a smooth connected base $Y$.
    Further assume that the following conditions:
    \begin{itemize}
        \item for every $y \in Y$, we have $(X_y)_{\rm reg}=(X_{\rm reg})_y$ (which will then be denoted by $X_{y,{\rm reg}}$; and
        \item the fibers $X_y$ and their regular loci $X_{y,{\rm reg}}$ form topologically (in the Euclidean topology of $Y$) locally trivial families.
    \end{itemize}
    Then, if some closed fiber is a Calabi--Yau (resp. irreducible symplectic) variety, then all fibers are Calabi--Yau (resp. irreducible symplectic) varieties.
    Similarly, the augmented irregularity $\tilde{q}(X_y)$ is independent of $y \in Y$.
\end{proposition}

\begin{proof}
    Let $U$ be a contractible Euclidean open set such that $X_U \coloneqq X \times_Y U$ and $X_{U,{\rm reg}} \coloneqq X_{\rm reg} \times_Y U$ are isomorphic as topological manifolds to $X_{y_0} \times U$ and $X_{y_0,{\rm reg}} \times U$, respectively, where $y_0$ is an arbitrary point of $U$.
    In particular, $X_U$ (resp $X_{U,{\rm reg}}$) deformation retracts to $X_{y_0}$ (resp. $X_{y_0,{\rm reg}}$).
    In particular, we get identifications of the corresponding fundamental groups.
    
    Let $X'_{y_0} \to X_{y_0}$ be a quasi-\'etale cover.
    Then, it corresponds to a finite subgroup of $\pi_1(X_{y_0,{\rm reg}})$.
    By the isomorphism $X_{U,{\rm reg}} \simeq X_{y_0,{\rm reg}} \times U$, this induces an \'etale cover $X'_{U,{\rm reg}} \to X_{U,{\rm reg}}$ in the analytic category.
    By the Grauert--Remmert Extension Theorem \cite[Ch.~XII, Thm.~5.4, p.~340]{SGA1}, this cover extends to a quasi-\'etale cover of $X'_U \to X_U$ in the analytic category.
    In turn, by the Riemann existence theorem \cite[Ch.~XII, Thm.~5.1, p.~333]{SGA1}, $X'_y \to X_y$ is a quasi-\'etale cover in the algebraic category for every $y \in U$.
    By assumption, $f$ is flat, so by \cite[Ch.~XII, Prop.~3.1, p.~321]{SGA1}, also its analytification is flat.
    In particular, $X_U \to U$ is flat and hence open in the Euclidean topology by \cite[Ch.~V, Thm.~2.12]{Banica}.
    Also, $X'_U \to X_U$ is open by \cite[Ch.~V, Lem.~3.2]{Banica}.
    Hence, the map $X'_U \to U$ is open.
    In turn, by \cite[Ch.~V, Thm.~2.13]{Banica},
    $X'_U \to U$ is flat and by \autoref{lemma:constant_diff_forms} $h^0(X'_y,\Omega_{X'_y}^{[p]})$ is independent of $y \in U$.

    Since $y_0$ is an arbitrary point of $U$, this shows that the quasi-\'etale covers of any two points of $U$ are in natural one-to-one correspondence.
    Furthermore, this correspondence preserves the dimension $h^0(X'_y,\Omega_{X'_y}^{[p]})=h^p(X'_y,\mathcal{O}_{X'_y})$.
    Therefore, since being Calabi--Yau  (resp. irreducible symplectic) is characterized by the values $h^0(X'_y,\Omega_{X'_y}^{[p]})$ for arbitrary quasi-\'etale covers, it follows that both being and not being Calabi--Yau (resp. irreducible symplectic) are open properties in the Euclidean topology of $Y$.
    Similarly, the augmented irregularity is defined as the maximum $h^1(X'_y,\mathcal{O}_{X'_y})$ among all quasi-\'etale covers, and therefore the augmented irregularity is locally constant.
    Then, since $Y$ is connected, the claim follows.
\end{proof}

As recalled in \autoref{subsection_log_cy}, a klt $K$-trivial variety admits a quasi-\'etale cover that splits as a product of Calabi--Yau, abelian, and irreducible symplectic varieties.
More generally, a klt log Calabi--Yau pair admits a quasi-\'etale cover that splits as a product of a rationally connected klt log Calabi--Yau pair and Calabi--Yau, abelian, and irreducible symplectic varieties;
see \cite{MW25}.
We call any such cover a {\it Beauville--Bogomolov cover}.

\begin{proposition}\label{bb_in_family}
    Let $f \colon (X,B) \to T$ be a locally stable family of klt log Calabi--Yau pairs over a normal base $T$.
    Then, up to a locally closed decomposition of $T= \sqcup_i T_i$ and finite covers $V_i \to T_i$, the families $(X_{V_i},B_{V_i}) \to V_i$ obtained by base change admit quasi-\'etale covers that induce a Beauville--Bogomolov cover fiberwise.
    In parituclar, the augmented irregularity of a closed fiber over $t \in V_i$ coincides with the augmented irregularity of the generic and geometric generic fibers of $X_{V_i} \to V_i$.
\end{proposition}

\begin{proof}
    To begin, we may stratify $T= \sqcup_i T_i$ so that, over each component $T_i$, the fibers $X_t$ and their regular loci $(X_t)_{\rm reg}$ form topologically locally trivial families;
    see, e.g., \cite[\S~1.5-1.7]{GMP88}, or \cite[Lemma 1.5.A]{Nori83} for the regular loci $(X_t)_{\rm reg}$.
    Up to a further stratification, we may also assume that each $T_i$ is smooth.
    We observe once and for all that, by Noetherian induction, we may shrink $T_i$ and refine the stratification whenever necessary.

    Recall that, for every $T_i$, the geometric generic fiber is isomorphic, as a scheme over $\mathbb Z$, to a very general fiber.
    Said otherwise, the two varieties are identified by an automorphism of $\mathbb C$;
    see, e.g., \cite[Lemma~21]{ST19}.
    Notice that this procedure also identifies the regular loci and their \'etale covers.
    Therefore, the set of quasi-\'etale covers for some $X_t$ with $t \in T_i$ very general coincides with the set of quasi-\'etale covers of the geometric generic fiber $X_{\overline{k(T_i)}}$ over $T_i$.

    Now, by \cite{MW25} and the Lefschetz principle, $(X_{\overline{k(T_i)}},B_{\overline{k(T_i)}})$ admits a Beauville--Bogomolov cover $\widetilde X_{\overline{k(T_i)}} \to  X_{\overline{k(T_i)}}$, which is then defined over a finite field extension of $k(T_i)$.
    Let $V_i \to T_i$ the corresponding finite cover;
    notice that, up to shrinking $T_i$ and refining the stratification of $T$, we may further assume that $V_i$ is smooth and $V_i \to T_i$ \'etale.
    Up to further shrinking $T_i$ and $V_i$, we may assume that the product structure of the cover $\widetilde X_{k(V_i)} \to X_{k(V_i)}$ spreads out over $V_i$ and that this gives a locally stable family of log Calabi--Yau pairs.
    To summarize, up to an alteration $V_i \to T_i$ and possibly refining the original stratification of $T$, we construct a quasi-\'etale cover of the generic fiber of the family $\widetilde X_{k(V_i)} \to X_{k(V_i)}$;
    then, this cover spreads out over $V_i$ and, by specialization, gives a quasi-\'etale cover of each closed fiber.
    What is left to verify is that, over any closed fiber, the induced cover is indeed a Beauville--Bogomolov cover.

    Now, we observe the following facts:
    \begin{itemize}
        \item[(a)] since the projective line is defined over $\mathbb Z$, being rationally connected is invariant under base change by a field automorphism of proper varieties;

        \item[(b)] the canonical sheaf and the structure sheaf are invariant under base change by a field automorphism;

        \item[(c)] the dimension of cohomology groups are invariant under base change by a field automorphism; and

        \item[(d)] a field automorphism transforms a quasi-\'etale cover into a quasi-\'etale cover.
    \end{itemize}
    In particular, being a log Calabi--Yau pair is invariant under field automorphism.
    Furthermore, being a Calabi--Yau or irreducible symplectic variety is characterized by the dimension of certain cohomology groups of quasi-\'etale covers and is then an invariant property under field automorphism.
    Then, the same holds true for abelian varieties;
    see, e.g., \cite{CH01}.
    In particular, for a very general closed point $t \in V_i$, an abelian (resp. Calabi--Yau, resp. irreducible symplectic, resp. geometrically rationally connected) factor of $\widetilde X _{k(V_i)}$ specializes to an abelian (resp. Calabi--Yau, resp. irreducibly symplectic, resp. rationally connected) factor of $\widetilde X_{V_i,t}$.
    Then, it follows that being a Beauville--Bogomolov cover is invariant under field automorphisms.
    
    In turn, the chosen Beauville--Bogomolov cover $\widetilde X_{V_i} \to  X_{V_i}$ induces a Beauville--Bogomolov cover for a very general point of $V_i$.
    To conclude, we need to assess that this is the case for every closed point in $V_i$.
    That is, we have to rule out that, for some $t_0 \in V_i$, some factor of the product decomposition of $\widetilde X_{V_i,t_0}$ is not of one of the allowed types:
    rationally connected, abelian, Calabi--Yau, or irreducible symplectic.
    By \cite[Corollary~IV.3.5]{Kol96} and \cite{HM07}, the rationally connected factor specializes to a rationally connected factor for every point in $V_i$.
    Similarly, by \cite{CH01} and deformation invariance of plurigenera, the abelian factor specializes to an abelian factor.
    In turn, the only possible source of contradiction is that some Calabi--Yau or irreducible symplectic factor does not specialize to such.
    Then, also Calabi--Yau or irreducible symplectic factors specialize to factors of the same type by \autoref{prop:def_CY_IS}, hence the claim follows.
\end{proof}

\section{Main technical results}

In this section, we collect the main technical statements of this work.

\subsection{$K$-trivial orbibundles in low dimensions}\label{subsection_orbibundle}

The main goal of this work is to characterize elliptic Calabi--Yau 4-folds.
In view of \autoref{lemma:orbibundle_structure}, Calabi--Yau 4-folds that are crepant birational to a $K$-trivial orbibundle can be treated separately from the general case.
Thus, in this subsection, we collect properties of this special class of Calabi--Yau 4-folds.

\begin{proposition}\label{prop:orbibundle}
Let $(\widehat Y \times F)/G \to \widehat Y /G \eqqcolon Y$ be an orbibundle such that $F$ is an elliptic curve and $Y$ is a $K$-trivial $(n-1)$-fold.
Further assume that a  crepant model of $(\widehat Y \times F)/G$, say $X$, is a Calabi--Yau $n$-fold. 
We denote the index one cover of $Y$ by $Y^\prime$. 
Then, the following facts hold:
\begin{itemize}
    \item there exists a factorization $\widehat Y \to Y^\prime \to Y$; 
    \item $\widehat Y$ and $Y^\prime$ are canonical $K$-trivial $(n-1)$-folds;
    \item $Y$ is a rationally connected klt variety; and
    \item the index of $K_Y$ is 2, 3, 4, or 6. 
\end{itemize}
Furthermore, if the augmented irregularity $\tilde q (Y)$ vanishes and $n=4$, $\widehat Y$ and $Y^\prime$ are Calabi--Yau 3-folds.

Conversely, given a rationally connected klt variety $Y$ such that $K_Y \equiv 0$
and the index of $K_Y$ is 2, 3, 4, or 6, one can construct a canonical $K$-trivial orbibundle $X$ with linearly trivial canonical class.
Furthermore, 
the orbibundle $X$ is simply connected.
In particular, if $n=4$, $\tilde q (Y)$ vanishes, and $X$ admits a crepant resolution $X' \to X$, $X'$ is a Calabi--Yau 4-fold.
\end{proposition}

\begin{remark}
The assumption that $\tilde q (Y)$ vanishes is not artificial. In fact, when we have $n=4$, $\tilde q (Y)>0$, the set of such Calabi-Yau 4-fold $X$ is bounded by \autoref{main_thm_technical}. 
\end{remark}

\begin{proof}
The variety $Y$ is rationally connected by \autoref{prop:rc} and of klt type by the canonical bundle formula, see \autoref{rmk:cbf}.
By definition of an orbibundle, we know that $\widehat Y \times F \to (\widehat Y \times F)/G$ is quasi-\'etale. 
So the canonical divisor of $\widehat Y \times F$ is numerically trivial. 
Similarly, since $X$ admits a global $n$-form, so does $(\widehat Y \times F)/G$, which implies $\widehat Y \times F$ also admits a global $n$-form. 
But $K_{\widehat Y \times F} = \pi_1^*K_{\widehat Y}$, so $K_{\widehat Y}$ is linearly trivial. 
By $K_{(\widehat Y \times F)/G} \sim 0$, it follows that $G$ acts trivially on $H^0(\widehat Y \times F,\omega_{\widehat Y \times F}) = H^0(\widehat Y,\omega_{\widehat Y})\otimes H^0( F,\omega_{F})$. 

Let $G^\prime$ be the subgroup of $G$ consisting of elements which act trivially on $H^0(F,\omega_F)$. 
Then $G^\prime$ also acts trivially on $H^0(\widehat Y,\omega_{\widehat Y})$. 
So $Y^\prime = {\widehat Y}/G^\prime$ is a canonical $K$-trivial 3-fold with $K_{Y^\prime} \sim 0$. 
Moreover, $Y = Y^\prime/(G/G^\prime)$ and $G/G^\prime$ acts faithfully on $H^0(Y^\prime,\omega_{Y^\prime})$, so $Y^\prime$ is the index $1$ cover of $Y$.
By the classification of automorphisms of elliptic curves and the fact that translations act trivially on $H^0(F,\omega_F)$, $|G/G^\prime|$ is equal to 1, 2, 3, 4, or 6;
see, e.g., \cite[Corollary IV.4.7]{Har77}.
Since $Y$ is rationally connected with $K_Y \equiv 0$,
$Y$ is klt but not canonical, so the index of $K_Y$ is at least 2. 
So the index of $K_Y$ is 2, 3, 4, or 6.

Now, if $\tilde q (Y) = 0$ and $n=4$, then we have $\tilde q(Y') = \tilde q (\widehat Y) = 0$.
Then, for any quasi-\'etale cover $Y''$ of one of these varieties, we have $h^1(Y'',\mathcal{O}_{Y''})=0$ and, by Serre duality, $h^2(Y'',\mathcal{O}_{Y''})=0$.
In turn, by \cite[Theorem 1]{Sch16}, we have $h^0(Y'',\Omega_{Y''}^{[i]})=0$ for $i=1,2$ and it follows that $Y'$ and $\widehat Y$ are Calabi--Yau.

Conversely, given such $Y$, we construct the desired orbibundle.
Let $d$ be the index of $K_Y$. 
Let $\widehat Y \to Y$ be the global index 1 cover with Galois group $G = \mathbb{Z}/d\mathbb{Z}$. 
Choose an elliptic curve $F$ with a faithful action of $G$ inducing a faithful action on $H^0(F,\omega_F)$.
Possibly composing with a translation, we may assume the action of $G$ on $F$ admits a fixed point $P_0$;
see, e.g., \cite[Corollary IV.4.7]{Har77}.
Possibly replacing the $G$-action on $F$ by the inverse, we may assume $G$ acts trivially on $H^0(\widehat Y \times F,\omega_{\widehat Y \times F}) = H^0(\widehat Y,\omega_{\widehat Y})\otimes H^0( F,\omega_{F})$. 
Note that the action of $G$ on $\widehat Y \times F$ is free in codimension $1$. 
So, we have $h^0((\widehat Y \times F)/G, \omega_{(\widehat Y \times F)/G}) = 1$, i.e., $K_{(\widehat Y \times F)/G}$ is linearly equivalent to $0$.
Moreover, by the ramification formula, $(\widehat Y \times F)/G$ is klt, so $(\widehat Y \times F)/G$ must be canonical. 

Let $\widetilde Y \to \widehat Y$ be the universal cover with Galois group $\Gamma$.
Then, there is a 
possibly infinite
group $\widetilde G$, which is an extension of $G$ by $\Gamma$, that acts on $\widetilde Y$ so that $\widetilde Y / \widetilde G = \widehat Y / G = Y$.
By construction, $\widetilde G$ acts effectively on $\widetilde Y$.
Thus, $\widetilde G$ acts effectively on $\widetilde Y \times F$, where the action on $F$ factors through $\widetilde G \to G$.
In particular, we have $(\widehat Y\times F)/G = (\widetilde Y \times F)/\widetilde G$.

Topologically, we have $(\widetilde Y\times F)/ \widetilde G = (\widetilde Y \times \mathbb{C})/( \mathbb{Z}^2 \rtimes \widetilde G)$.
By \cite{Arm68}, to show that $(\widetilde Y\times F)/ \widetilde G$ is simply connected, it suffices to show $ \mathbb{Z}^2 \rtimes \widetilde G$ is generated by elements with fixed points.
This is 
addressed in \autoref{lemma:fixed_points}.

To conclude, assume that $n=4$, $\tilde q (Y)$ vanishes, and the orbibundle $X$ admits a crepant resolution $X' \to X$.
Since $X$ is simply connected, then so is $X'$ by \cite{Tak03}.
Thus, we have $h^1(X',\mathcal{O}_{X'})=0$ and, by Serre duality, $h^3(X',\mathcal{O}_{X'})=0$.
Thus, to conclude, it suffices to show $h^2(X,\mathcal{O}_X)=0$.
By the vanishing of the augmented irregularity of $Y$, we have $H^1(\widehat Y,\mathcal{O}_{\widehat Y})$.
Since $\widehat Y$ is a canonical 3-fold, by Serre duality and $K_{\widehat Y} \sim 0$ we also get $H^2(\widehat Y,\mathcal{O}_{\widehat Y})=0$.
Thus, by the K\"unneth formula and the vanishing of the groups $H^i(\widehat Y,\mathcal{O}_{\widehat Y})$
for $i = 1,2$ and $H^2(F,\mathcal{O}_F)$, we have
$H^2(\widehat Y \times F,\mathcal{O}_{\widehat Y \times F}) = 0$.
In turn, by considering $G$-invariants under the $G$-action, we obtain $H^2(X,\mathcal{O}_X)=0$, and the claim follows.
\end{proof}

\begin{lemma}\label{lemma:fixed_points}
    Let $(\widetilde Y \times F)/\widetilde G = (\widetilde Y \times \mathbb C)/(\mathbb{Z}^2\rtimes \widetilde G )$ be as in the proof of \autoref{prop:orbibundle}.
    Then,
    the normal subgroup of $\mathbb{Z}^2 \rtimes \widetilde G$ generated by elements with fixed points is the whole group.
\end{lemma}

\begin{proof}
Recall that $\mathbb{Z}^2 \rtimes \widetilde G$ acts on $\widetilde Y \times \mathbb C$ componentwise.
Thus, fix $g\in \mathbb Z ^2 \rtimes \widetilde G$ and let $g^\prime$ be its image in $\mathbb{Z}^2 \rtimes G$.
By construction, $G$ acts effectively on the elliptic curve $F$, so
$\mathbb{Z}^2 \rtimes G$ acts effectively on $\mathbb{C}$.
Observe that if $ g^\prime\notin \mathbb{Z}^2$, then the action of $g$ on $\mathbb{C}$ admits a fixed point.
Indeed, $g^\prime$ can be written as a composition $g^\prime = r \circ t$, where $t\in \mathbb{Z}^2$ is a translation and $r$ is given by multiplication of some $\alpha \in \mathbb{C}$ by \cite[Proposition IV.4.18]{Har77}.
Since $g^\prime\notin \mathbb{Z}^2$, $r_0 \neq 1$, so $g^\prime - id_{\mathbb{C}}$ is bijective, and in particular then the action of $g$ on $\mathbb{C}$ admits a fixed point.

Then let $H$ be the normal subgroup of $\mathbb{Z}^2 \rtimes \widetilde G$ generated by elements with fixed points.
Recall that $Y$ is simply connected since it is rationally connected, see \autoref{prop:rc}, while $\widetilde Y$ is simply connected since it is the universal cover of $\widehat Y$.
Then, by \cite{Arm68},
the normal subgroup of $\widetilde G$ generated by elements with fixed points for the action on $\widetilde Y$ is $\widetilde G$. Since the action of $\widetilde G$ on $F$ always admits a fixed point, we have $\widetilde G\subset H$.

Finally, let $g$ be any nontrivial element in $\widetilde G$ which admits a fixed point for the action on $\widetilde Y$, and let $g^\prime$ be its image in
$\mathbb Z ^2 \rtimes G \subset {\rm Aut}(\mathbb{C})$.
Since $\Gamma$ acts on $\widetilde Y$ without fixed points, we have $g \not \in \Gamma$.
Thus, we have $g^\prime \notin \mathbb{Z}^2$ and, by the first part of the proof, $g^\prime \cdot x$ admits a fixed point for the action on $\mathbb C$ for every $x \in \mathbb Z^2$.
In turn, the action of $g \cdot x$ on $\widetilde Y \times \mathbb{C}$ admits a fixed point for any $x\in \mathbb{Z}^2$.
In particular, $g \cdot x\in H$ for any $x\in \mathbb{Z}^2$.
So we conclude that $H =\mathbb Z ^2 \rtimes \widetilde G$.
\end{proof}

\subsection{Boudnedness of the bases of elliptic Calabi--Yau 4-folds}

In this subsection, we prove the boundedness of the bases of elliptic Calabi--Yau 4-folds that are not crepant birational to a $K$-trivial orbibundle.

\begin{theorem}\label{log boundedness Iitaka dim > 0}
Fix $n,d \in \mathbb{N}_{>0}$, $d\leq 4$.
Let $V$ be the set of pairs $(X,\Delta)$ satisfying the conditions \autoref{condition:a}-\autoref{condition:b}-\autoref{condition:c} or \autoref{condition:a}-\autoref{condition:b}-\autoref{condition:c'} below:

\begin{enumerate}
    \item\label{condition:a}  $X$ is rationally connected of dimension $d$ and $\Delta$ is a $\mathbb{Q}$-divisor;

    \item\label{condition:b} $(X,\Delta)$ is klt and $n(K_X+\Delta)$ is integral and linearly equivalent to $0$;

    \item\label{condition:c} $\kappa(X',-K_{X'}) \geq d-2$, where $X' \to X$ is a small $\mathbb Q$-factorialization.

    \item\label{condition:c'} $d =3$, $\kappa(X',-K_{X'}) = 0$, where $X' \to X$ is a small $\mathbb Q$-factorialization, and $\tilde{q}(X) > 0$, where $\tilde{q}(X)$ denotes the augmented irregularity of $X$.

\end{enumerate}
Then, $V$ is log bounded. 
\end{theorem}

\begin{proof}
By \cite[Theorem 1.5]{BDCS24}, such pairs are log bounded in codimension $1$, i.e., we have a set $W$ of varieties such that for each $(X,B) \in V$, there exists some $(X^\prime,B^\prime) \in W$, such that $(X,B)$ and $(X^\prime,B^\prime)$ are isomorphic in codimension $1$, and $W$ is log bounded.
By \cite[Theorem 1.3]{Bir22}, the set $W^\prime$ of $\mathbb{Q}$-factorial terminal models of pairs in $W$ still form a log bounded family.
Observe that we use the notion of boundedness in a stronger sense than in \cite{Bir22}, which does not assume the control on the coefficients of the boundary;
nevertheless, we can claim boundedness in the sense of \autoref{def.bounded.set.pairs} by \cite[Proposition 2.10]{BDCS24}.

By definition, we have a projective morphism $h \colon (\mathcal{X},\mathcal{B}) \to \mathcal{Y}$, such that for each $(X,B)\in W^\prime$, there exists some $y\in \mathcal{Y}$, with $(X,B) \simeq (\mathcal{X}_y,\mathcal{B}_y)$. 
By taking a stratification of the base $\mathcal{Y}$, we may further assume $\mathcal{Y}$ is smooth and $h$ is flat.
By discarding some open subsets, we may assume the points $y\in \mathcal{Y}$ such that $(\mathcal{X}_y, \mathcal{B}_y) \in W^\prime$ are dense in $\mathcal{Y}$. 
For each $y\in \mathcal{Y}$ such that $(\mathcal{X}_y, \mathcal{B}_y) \in W^\prime$, there exists an \'etale neighborhood such that the results of \autoref{deformation lemma} hold;
notice that the required vanishing in cohomology holds since the pairs under consideration are rationally connected.
By \cite[Theorem~6.6]{LZ22}, \cite[Corollary~2]{FXu24} or \cite[Theorem~1.2]{FXu25} together with \autoref{bb_in_family}, after a further \'etale base change and noetherian stratification, we may assume the movable relative cone conjecture holds in this \'etale neighborhood. 
By noetherianity, we may consider only finitely many of these \'etale neighborhoods.

Thus, by passing to a stratification and an \'etale base change of the parameter space $\mathcal Y$, we may assume $h$ is a finite union of $h_i \colon (\mathcal{X}_i,\mathcal{B}_i) \to \mathcal{Y}_i$, and each $h_i$ is a terminal $\mathbb{Q}$-factorial log Calabi--Yau fiber space such that the relative cone conjecture holds. 
Furthermore, up to a stratification, we may assume that each $(\mathcal{X}_i,\mathcal{B}_i)$ admits a log resolution that is relatively log smooth over $\mathcal{Y}_i$; in particular, this log resolution restricts to a log resolution fiberwise.
Since each $\mathcal{X}_i$ is $\mathbb Q$-factorial and $h_i$ is projective, by \cite[Theorem-Definition 4.3]{kollar_moduli}, each $h_i$ is a well-defined family of pairs in the sense of {\it op. cit.}.
Then, by the existence of the fiberwise log resolution, by \cite[Theorem 4.54]{kollar_moduli}, each $h_i \colon (\mathcal{X}_i,\mathcal{B}_i) \to \mathcal{Y}_i$ is a locally stable family of pairs; see \cite[Theorem-Definition 4.7]{kollar_moduli}.

Now for each $h_i$, there are only finitely many small $\mathbb{Q}$-factorial modifications, and there are only finitely many birational contractions starting from them. 
Let $\bar{h}$ be the disjoint union of all such small $\mathbb{Q}$-factorial modifications and birational contractions starting from them. 
We claim that every pair $(X,B) \in V$ appears in the family $\bar{h}$. 

In order to better exploit \autoref{deformation lemma}, we distinguish two cases.

{\bf Case 1:} In this step, we assume that $(X,B)$ is terminal $\mathbb Q$-factorial.

By assumption, $(X,B)$ is isomorphic in codimension 1 to some $(X^\prime,B^\prime) \in W^\prime$, say $f: (X^\prime,B^\prime) \dashrightarrow (X,B)$.
Then $(X^\prime,B^\prime)$ is a fiber of $h_i \colon (\mathcal{X}_i,\mathcal{B}_i) \to \mathcal{Y}_i$, say $(X^\prime,B^\prime) = (\mathcal{X}_{i,y},\mathcal{B}_{i,y})$. 
Pick an ample divisor $A$ on $X$, and consider the pullback $f^*(A)$ on $X^\prime$.
Since $X$ and $X^\prime$ are isomorphic in codimension 1 and are both $\mathbb Q$-factorial, $f^*(A)$ is movable.
Furthermore, since $A$ is ample, it is in the interior of the movable cone of $X$, and so is $f^*(A)$ in the interior of $\mathcal{M}(X^\prime)$.
By \autoref{deformation lemma} and the full-dimensionalty of the cones involved, $f^*(A)$ deforms to a big and movable divisor $\mathcal{D}$ on $\mathcal{X}_i$.

We run an MMP with scaling on $(\mathcal{X}_i,\mathcal{B}_i+\epsilon \mathcal{D})$ over $\mathcal{Y}_i$, where $0 < \epsilon \ll 1$.
Since $\mathcal{D}$ is big and movable, this MMP consists of log flips and terminates with a good minimal model $(\mathcal{X}_i^\prime,\mathcal{B}_i^\prime+\epsilon \mathcal{D}^\prime)$ over $\mathcal{Y}_i$ by \cite{BCHM10}. 

Now, we argue that we can apply \cite[Lemma 3.1]{HMX18}.
By assumption, $(\mathcal{X}_{i,y},\mathcal{B}_{i,y})$ is terminal, so condition (3) of {\it op. cit.} is satisfied.
Similarly, $(\mathcal{X}_{i},\mathcal{B}_i)$ is terminal and $h_i$ is equidimensional, thus condition (2) is satisfied.
Lastly, since $h_i$ is a locally stable family, by \cite[Theorem 4.54]{kollar_moduli}, we may choose general hyperplanes $H_1, \ldots, H_{\dim \mathcal{Y}_i}$ through $y \in \mathcal{Y}$ such that $(\mathcal{X}_{i},\mathcal{B}_i+\sum P_j)$ is log canonical, where $P_i$ denotes the pullback of $H_j$ to $\mathcal{X}_i$.
Furthermore, by construction, $h_i$ admits a fiberwise log resolution.
Thus, there exists a closed subset $\mathcal{Z}_i$ of codimension at least 2 and with no vertical component over $\mathcal{Y}_i$ such that $(\mathcal{U}_i,\Supp(\mathcal{B}_i|_{\mathcal{U}_i})) \to \mathcal{Y}_i$ is a log smooth family.
In particular, the pair $(\mathcal{X}_{i},\mathcal{B}_i+\sum P_j)$ is log smooth at the generic point of each stratum of $\sum P_j$.
To conclude that $(\mathcal{X}_{i},\mathcal{B}_i+\sum P_j)$ is dlt, we then only have to show that the strata of $\sum P_j$ are the only log canonical centers of $(\mathcal{X}_{i},\mathcal{B}_i+\sum P_j)$.
Assume by contradiction that there exists another log canonical center $C$.
To apply \cite[Lemma 3.1]{HMX18}, we are only concerned with our reference point $y \in \mathcal{Y}_i$.
So, we may assume $C \cap \mathcal{X}_{i,y} \neq \emptyset$.
Then, by \cite[Corollary 4.41]{Kol13}, $\mathcal{X}_{i,y}$ properly contains a log canonical center of $(\mathcal{X}_{i},\mathcal{B}_i+\sum P_j)$.
This is impossible, since $(\mathcal{X}_{i,y},\mathcal{B}_{i,y})$ is terminal and it is the pair obtained by adjunction from $(\mathcal{X}_{i},\mathcal{B}_i+\sum P_j)$ by \cite[Theorem 4.54]{kollar_moduli}.
Thus, condition (1) in {\it op. cit.} is satisfied.

By \cite[Lemma 3.1]{HMX18}, each step of this MMP restricts to a $f^*(A)$-negative birational contraction of the fiber over $y$. 
Since $f^*(A)$ is movable, and $(\mathcal{X}_{i,y},\mathcal{B}_{i,y}+\epsilon f^*(A))$ is terminal as $\epsilon$ is sufficiently small, the $f^*(A)$-negative birational contractions must be small.
So $(\mathcal{X}_{i,y}^\prime,\mathcal{B}_{i,y}^\prime+\epsilon \mathcal{D}_y^\prime)$ is isomorphic in codimension $1$ to $(X^\prime,B^\prime+\epsilon f^*(A))$, and $\mathcal{D}_y^\prime$ is nef and big.
Moreover, the image of the Iitaka fibration of $\mathcal{X}_{i,y}^\prime$ induced by $\mathcal{D}_y^\prime$ is precisely $X$.
Consider the Iitaka fibration of $\mathcal{X}_i^\prime/\mathcal{Y}_i$ induced by $\mathcal{D}^\prime$, say $g \colon  \mathcal{X}_i^\prime/\mathcal{Y}_i 
\to \mathcal{X}_i^{\prime,amp}/\mathcal{Y}_i $, and $g^*\mathcal{D}^{\prime,amp} = \mathcal{D}^\prime$. 
By \autoref{lemma commutativity of specialization},  $((\mathcal{X}_i^{\prime,amp})_y,(g_*\mathcal{B}^\prime_i)_y) \simeq (X,B)$.
\footnote{We observe that, under the assumptions of Case 1, $(\mathcal{X}_{i,y}^\prime,\mathcal{B}_{i,y}^\prime+\epsilon \mathcal{D}_y^\prime)$ is actually isomorphic to $((\mathcal{X}_i^{\prime,amp})_y,(g_*\mathcal{B}^\prime_i)_y)$, so taking the ample model is not necessary.
On the other hand, this will no longer be true in Case 2, and that is why we include this step in the proof.}

So every pair $(X,B) \in V$ that is also terminal and $\mathbb Q$-factorial appears in $\bar{h}$.

{\bf Case 2:} In this step, we treat the general case assuming that Case 1 has been settled.
The only difference is that we cannot invoke \autoref{deformation lemma} immediately, while the subsequent part of the proof remains identical and is therefore omitted.

By Case 1, we may fix a $\mathbb Q$-factorial terminalization $(X^\prime,B^\prime)$ of $(X,B)$,
and this terminalization belongs to a bounded family.
Suppose $(X^\prime,B^\prime)$ is a fiber of $h_i \colon (\mathcal{X}_i,\mathcal{B}_i) \to \mathcal{Y}_i$, say $(X^\prime,B^\prime) = (\mathcal{X}_{i,y},\mathcal{B}_{i,y})$. 
Pick an ample divisor $A$ on $X$, and consider the pullback $f^*(A)$ on $X^\prime$.
In general, $f^*(A)$ may lie in the boundary of $\bar{\mathcal{M}}(X^\prime)$;
this is for instance the case if
the morphism $X^\prime \to X$ is not small.
By \autoref{deformation lemma}, $f^*(A)$ deforms to a movable divisor $\mathcal{D}$ on $\mathcal{X}_i$.
Unlike in the previous case, we cannot invoke \autoref{deformation lemma} to conclude that $\mathcal{D}$ is also big.
Yet, since $X^\prime \to X$ is a morphism and $A$ is ample, $\mathcal{D}$ restricts to a nef divisor to the generic fiber of the family.
Then, by the deformation invariance of the self-intersection of $f^*(A)$, we conclude that said restriction to the generic fiber is also big.
Thus, we deduce that $f^*A$ deforms to a big and movable divisor $\mathcal{D}$ on $\mathcal{X}_i$.
Furthermore, while $\mathcal{D}$ is a priori only in the closure of $\bar{\mathcal{M}}(\mathcal{X}_i/\mathcal{Y}_i)$ of $\mathcal{M}(\mathcal{X}_i/\mathcal{Y}_i)$, it is nevertheless limit of divisors in $\mathcal{M}(\mathcal{X}_i/\mathcal{Y}_i)$; in particular, any MMP for $\mathcal{D}$ is also an MMP for a divisor in $\mathcal{M}(\mathcal{X}_i/\mathcal{Y}_i)$ and thus it only consists of small maps.
Once this is settled, the proof continues as in Case 1, and we may conclude.
\end{proof}

\begin{corollary}\label{cor_bdd_bases}
Let $V$ be the set of varieties $X$ such that 
\begin{enumerate}
    \item  there exists an elliptic fibration $Y \to X$ where $Y$ is a Calabi--Yau 4-fold; and

    \item $\kappa(X^\prime,-K_{X^\prime}) \geq 1$ or $\tilde q (X^\prime)>0$, where $X^\prime$ is a small $\mathbb{Q}$-factorization of $X$.
\end{enumerate}
Then $V$ is bounded. 
\end{corollary}

\begin{proof}
By \autoref{prop:rc}, $X$ is rationally connected.
By \cite[Theorem 5.4]{BDCS24}, there exists a universal constant $n$ such that there exists an effective divisor $\Delta > 0$ on $X$, and $(X, \Delta)$ is a klt
with $n(K_X + \Delta) \sim 0$. 
Then, $V$ is bounded by \autoref{log boundedness Iitaka dim > 0}. 
\end{proof}

\section{Proof of the main statements}

In this section, we collect the proof of the main results.

\subsection{Boundedness results}

\begin{theorem}\label{main_thm_technical}
Fix $n \in \mathbb{N}_{>0}$.
Let $V$ be the set of varieties $X$ such that 
\begin{enumerate}
    \item $X$ is an elliptic Calabi--Yau variety of dimension $4$;
    \item \label{main_thm_technical_part_b} $f \colon X \to Y$ is not crepant birational to an orbibundle whose base has vanishing augmented irregularity.
\end{enumerate}
Then, $V$ is bounded.
\end{theorem}

\begin{proof}
    It follows from \autoref{cor_bdd_bases}, \cite[Theorem~A]{EFGMS}, and \cite[Theorem 6.18]{FHS21}.
    We observe that \cite[Theorem~6.18]{FHS21} requires $X$ to be terminal and $\mathbb Q$-factorial;
    we can circumvent this technical assumption by arguing as in Case 2 of the proof of \autoref{log boundedness Iitaka dim > 0}.
\end{proof}

\begin{remark}\label{rmk_cy}
    By the Beauville--Bogomolov decomposition for klt $K$-trivial varieties and the classification of surfaces, a $K$-trivial 3-fold $Y$ has vanishing augmented irregularity if and only if it has a quasi-\'etale cover $Y' \to Y$ where $Y'$ is a Calabi--Yau 3-fold.
\end{remark}

\subsection{Moduli part results}

Recall that, by \autoref{lc-trivial.def}, when considering an lc-trivial fibration $f \colon (X,\Delta) \to Y$, we assume $\mathrm{coeff}(\Delta) \subset \mathbb Q$.

\begin{proposition}\label{prop_cartier_index_lc}
    Let $f \colon (X,\Delta) \to Y$ be an lc-trivial fibration such that
    \begin{itemize}
        \item $\dim X - \dim Y =3$;
        \item $\Delta$ is effective over the generic point of $Y$ and $\mathrm{coeff}(\Delta)\subset \mathbb Q$; and
        \item $(X,\Delta)$ is log canonical but not klt over the generic point of $Y$.
    \end{itemize}
    Let $(Y,B_Y,\mathbf M)$ denote the generalized pair induced on $Y$.
    Then,  there is a constant $I$, only depending on the horizontal multiplicities of $\Delta$, such that $I \cdot \mathbf M$ is b-Cartier. 
\end{proposition}

\begin{proof}
    Since the moduli part of the canonical bundle formula only depends on the crepant birational class of the generic fiber, by \cite{HX13} or \cite{Bir12}, we may first replace $(X_\eta,\Delta_\eta)$ with a $\mathbb Q$-factorial dlt modification and then choose a compactification of $(X_\eta,\Delta_\eta)$ over a projective model of $Y$ so that $(X,\Delta)$ is itself a $\mathbb Q$-factorial projective dlt pair and, after possibly replacing $Y$ with a projective compactification and a birational model, with $K_X + \Delta \sim_{\mathbb Q,Y} 0$.
    Furthermore, by this reduction, we may assume that every divisor in $\mathrm{Supp}(\Delta)$ dominates $Y$.
    In this process the only possible change in the coefficients of $\Delta$ is the introduction of components with coefficient 1;
    therefore, this reduction step does not alter the statement to prove.

    By the above reduction and the fact that $(X_\eta,\Delta_\eta)$ is not klt, $\Delta^{=1}$ dominates $Y$, where we write $\Delta=\Delta^{<1} + \Delta^{=1}$ and $\Delta^{=1}$ denotes the components with coefficient 1.
    Thus, we may run a $(K_X+\Delta^{<1})$-MMP with scaling over $Y$, which ends with a Mori fiber space $X' \rar Z$.
    If $Z$ is birational to $Y$, we conclude by \cite[Proposition~6.3]{Bir19}.
    Otherwise, we have a factorization $(X',\Delta') \rar Z \rar Y$, where $X' \to Z$ and $Z \to Y$ are fibrations of relative dimensions 1 or 2.

    By \cite[Theorem C]{EFGMS} and \cite[Theorem 1.4]{babwild}, the effective b-semiampleness conjecture \cite[Conjecture 7.13.3]{PS09} is known in relative dimension 2.
    Therefore, by \cite[Theorem 7.6]{Fil20} and its proof, effective b-semiampleness in relative dimension 3 holds true if the fibration factors as two fibrations of positive relative dimension.
    In particular, this grants the effectivity of the b-Cartier index of the moduli part.
    This concludes the proof.
\end{proof}

\begin{theorem}\label{thm_cartier_index}
    Let $f \colon (X,\Delta) \to Y$ be an lc-trivial fibration such that
    \begin{itemize}
        \item $\dim X - \dim Y =3$;
        \item $\Delta$ is effective over the generic point of $Y$ and $\mathrm{coeff}(\Delta)\subset \mathbb Q$; and
        \item $(X,\Delta)$ is 
        log canonical over the generic point of $Y$.
    \end{itemize}
    Let $(Y,B_Y,\mathbf M)$ denote the generalized pair induced on $Y$.
    Then,  there is a constant $I$, only depending on the horizontal multiplicities of $\Delta$, such that $I \cdot \mathbf M$ is b-Cartier. 
\end{theorem}

\begin{proof}
    By \autoref{prop_cartier_index_lc}, we may assume $(X,\Delta)$ is klt over the generic point of $Y$.

    It is well known that the statement can be verified under the additional assumption $\dim Y = 1$;
    see, e.g., \cite[\S~3]{FM00} and \cite[Lemma 3.1 and Proposition 5.2]{Flo12}.
    Furthermore, by taking a terminalization of the generic fiber $(X_\eta,\Delta_\eta)$ and spreading it out to a model over $Y$, we may assume that $(X_\eta,\Delta_\eta)$ is terminal.
    If $\Delta_\eta \neq 0$, we defer the proof to \autoref{thm:effective_freeness}.
    Thus, in the rest of the proof, we assume that the generic--and hence a general--fiber is a terminal $K$-trivial 3-fold.
    Lastly, since the moduli part of the canonical bundle formula only depends on the generic fiber, by \cite{HX13} or \cite{Bir12}, we may choose a different compactification of $X_\eta$ over $Y$ so that $X$ is itself a terminal $\mathbb Q$-factorial pair with $K_X \sim_{\mathbb Q,Y} 0$.

    We have the canonical bundle formula
\begin{equation}\label{eq_cbf}
I \cdot K_X \sim I \cdot f^*( K_Y+B_Y+\bM Y.),
\end{equation}
where $I$ denotes the index of the generic fiber.
By the index conjecture for 3-folds, see \cite[Theorem 1.8.7]{YXu}, $I$ belongs to a finite set.

Let $c$ be a closed point in $Y$. Let $t = {\rm lct}(X,0;f^*c)$. 
Then $(X,\Delta = tf^*c)$ is log canonical but not klt. 
Let $(T,\Delta_T)\to (X,\Delta)$ be a dlt modification and $E$ a prime divisor with coefficient 1 in $\Delta_T$. 
Then we may run $(K_T+\Delta_T-E)$-MMP over $Y$, which ends with a minimal model $f^\prime \colon X^\prime\to Y$ by \cite[Theorem 1.1]{Bir12}.
By the negativity lemma, the support of $f^{\prime*}c$ is the strict transform of $E$ on $X^\prime$, say $F$. Moreover, $\Delta_{X^\prime} = F$. 

Pick a log resolution $p \colon (Z,\Delta_Z)\to (X^\prime,\Delta_{X^\prime})$, where $K_Z+\Delta_Z = p^*(K_{X^\prime}+\Delta_{X^\prime})$.
Up to replacing $Z$ by a higher resolution, we may assume that $p$ factors as a composition $(Z,\Delta_Z) \to (W,\Delta_W) \to (X^\prime,\Delta_{X^\prime})$, where $(W,\Delta_W) \to (X^\prime,\Delta_{X^\prime})$ is a dlt modification that  only extracts divisors with log discrepancy $0$. 
We may write $\Delta_Z = S + B $ and $\Delta_W = S_W $, where $S$ (resp. $S_W$) is a reduced divisor, and the coefficients of $B$ lie in $(-\infty,1)$.
Moreover, we may assume $Z\to W$ does not contract any component of $S$ by the existence of thrifty resolutions, see \cite[Corollary 1.36]{Kol13}.

By adjunction, we have $(K_W+S_W)|_{S_W} \sim_{\mathbb{Q}} K_{S_W}+B_{S_W}$
and ${\rm coeff}(B_{S_W}) \in \Phi$, where $\Phi$ is the standard set, see \cite[Corollary 16.7]{Cor92}.
By \cite[Remark 1.2(3)]{Fuj00}, $(S_W,B_{S_W})$ is semi log canonical. By \cite[Theorem 1.8.8]{YXu}, the index conjecture holds for semi log canonical 3-folds, so there exists some natural number $m$ such that $m(K_{S_W}+B_{S_W}) \sim 0$. 

In the following, we use constructions similar to the ones in \cite[Proposition 8.1 Step 3]{Bir19}. 
Let $L = -mK_Z - mS - \lfloor (m+1)B \rfloor$. 
Consider the exact sequence
\begin{equation}\label{eq:ses}
0 \to \mathcal{O}_Z(L-S) \to \mathcal{O}_Z(L) \to \mathcal{O}_S(L|_S) \to 0.
\end{equation}
Note that $L-S = K_Z - \lfloor (m+1)(K_Z+\Delta_Z)\rfloor \sim_{\mathbb Q, Y} K_Z+\{(m+1)(K_Z+\Delta_Z)\} = K_Z+\{(m+1)\Delta_Z\}$.

By Koll\'ar's torsion freeness theorem, see \cite[Theorem 1.1]{Fuj09},  
$R^1(f^\prime\circ p)_*\mathcal{O}_Z(L-S)$
is torsion-free. 

In the following, $c$ still denotes a closed point of $Y$, and we adopt the notation of the previous paragraph for the birational models of $X$ obtained by modifications over $c$.

We distinguish the following two cases.
We observe that the cases do not depend on the choice of $c$, since the sheaf under consideration is torsion free and its stalk at the generic point of $Y$ is independent of $c$.

\noindent \textbf{Case 1. } 
$R^1(f^\prime\circ p)_*\mathcal{O}_Z(L-S) = 0$.

In this case, when restricted to an open neighborhood $U$ of $c$, we have $H^1(Z_U,\mathcal{O}_{Z_U}(L-S)) = 0$. 
Moreover, we claim that $H^0(S,L|_S)$ contains a section that is nonzero on every irreducible component of $S$.
Indeed, we may write $(K_Z+S+B)|_{S} \sim_{\mathbb{Q}} K_{S}+B_{S}$,
then $K_{S}+B_S$ equals the pullback of $K_{S_W}+B_{S_W}$ by \cite[Lemma 17.2.3]{Kol92}. So $-m(K_S+B_S) \sim 0$.
Moreover, note that we have \begin{equation}\label{eq:eq_adjunction}
L+m(K_Z+S+B) = mB-\lfloor (m+1)B\rfloor .
\end{equation}
By assumption, the support of $S+B$ is simple normal crossing, so $mB$ is integral near $S$. 
So, since the coefficients of $B$ are in $(-\infty,1)$, we have $mB-\lfloor (m+1)B\rfloor \geq 0$ near $S$. 
Since $\text{Supp}(mB-\lfloor(m+1)B\rfloor)$ does not contain any component of $S$, by restricting \eqref{eq:eq_adjunction} to $S$, we conclude that $L|_S$ admits a section that is nonzero on each component of $S$. 

By the long exact sequence in cohomology associated to \eqref{eq:ses} and $f \circ p$, and by the assumption of Case 1, $L$ admits a section over $U$, say $L|_U \sim D\geq 0$. Since this section is nonzero when restricted to any component of $S$, $D$ does not contain any component of $S$, hence $p_*D$ does not contain $F$. 

Then $-m(K_{X^\prime}+F) = p_*L  \sim p_*D = D^\prime \geq 0$
over some neighborhood $U$ of $c$. 
In this case, we have $D^\prime$ is vertical since $K_{X^\prime}$ is $\mathbb{Q}$-linearly trivial over the base, and so we may assume $D^\prime = 0$ by shrinking $U$ as $D^\prime$ does not contain $F$. This means that $m(K_{X^\prime}+F) \sim 0$ on $U$. Going back to $X$, we have
\begin{equation}\label{eq:linear_eq_cbf}
m(K_X+tf^*c) \sim 0 \; \textrm{on}\; U
\end{equation}
with $t = {\rm lct}(X,0;f^*c)$.  

Then we may bound the index of the moduli part using the canonical bundle formula, following \cite[Proposition 6.3]{Bir19}. 

It is known that $mK_{X_\eta} \sim 0$ by the index conjecture in dimension $3$. So $mK_{X} \sim V$ for some vertical divisor $V$. Since $V \sim_{\mathbb{Q},Y} 0$,
$V = f^{*}V_Y$ for some $\mathbb{Q}$-divisor $V_Y$ on $Y$. We choose the moduli part as $M_Y \coloneqq \frac{1}{m} V_Y - K_Y - B_Y$; 
see also \cite[7.5 Construction]{PS09}.
Then we have $mK_{X} \sim V =  f^{*}V_Y = mf^{*}(K_Y+B_Y+M_Y)$.

Let
$\Theta = \sum_{y\in {\rm Supp}(f (V))\cup {\rm Supp}(B_Y)} t_yf^{*}y$,
and $\Theta_Y = B_Y+ \sum_{y\in {\rm Supp}(f (V))\cup {\rm Supp}(B_Y)} t_yy$, where $t_y$ is the log canonical threshold of $f^{*}y$ with respect to $(X,0)$.

By $mK_X \sim V$ and the fact that $\Theta$ is supported on $f^{-1}(f(V))$, $m(K_X + \Theta)$ is linearly trivial over each point of $Y \setminus (\mathrm{Supp}(f(V))\cup \mathrm{Supp}(B_Y))$.
Then, by \eqref{eq:linear_eq_cbf}, 
$m(K_{X}+\Theta)$ is also linearly trivial over each point of ${\rm Supp}(f (V))\cup \mathrm{Supp}(B_Y)$.
Thus, by \cite[Lemma 2.4]{Bir19}, we deduce that $m(K_{X}+\Theta) \sim_{Y} 0$. Therefore we have $m(K_{X}+\Theta) = mK_{X}+m\Theta \sim mf^{*}(K_Y+B_Y+M_Y) + mf^{*}(\Theta_Y-B_Y) = mf^{ *}(K_Y+\Theta_Y+M_Y)$. So $mM_Y$ is integral as $K_Y+\Theta_Y$ is integral by definition. 

\noindent \textbf{Case 2. }  
$R^1(f^\prime\circ p)_*\mathcal{O}_Z(L-S)\neq 0$.

Since $R^1(f^\prime\circ p)_*\mathcal{O}_Z(L-S)$ is torsion-free, it is nonzero everywhere on $Y$.
Thus, a general fiber of $f$ has positive irregularity. 
By \cite[Theorem 5.1]{Flo12}, the b-Cartier index of $\bM Y.$ only depends on the middle Betti number and the index of a general fiber of $f$. Since a general fiber of $f$ is a canonical $K$-trivial threefold with positive irregularity, by \autoref{lemma:betti_number_obibundle}, the b-Cartier index of $\bM Y.$ is bounded.
\end{proof}

\begin{theorem}\label{thm:effective_freeness}
        Let $f \colon (X,\Delta) \to Y$ be an lc-trivial fibration such that
    \begin{itemize}
        \item $\dim X - \dim Y =3$; and
        \item $(X,\Delta)$ is klt over the generic point of $Y$ and $\mathrm{coeff}(\Delta)\subset \mathbb Q$; and
        \item if the generic fiber $(X_\eta,\Delta_\eta)$ is a canonical variety with $\Delta_\eta = 0$, then $X_\eta$ has positive augmented irregularity.
    \end{itemize}
    Let $(Y,B_Y,\mathbf M)$ denote the generalized pair induced on $Y$.
    Then, there is a constant $I$, only depending on the horizontal multiplicities of $\Delta$,
    such that $I \cdot \mathbf M$ is b-free.
\end{theorem}

\begin{proof}
    To start, if $(X_\eta,\Delta_\eta)$ is a canonical variety with $\Delta_\eta = 0$, we replace $X_\eta$ with a quasi-\'etale cover with positive irregularity and we take a compactification over $Y$.
    By definition of the canonical bundle formula, this does not alter the moduli b-divisor;
    see, e.g., \cite[Lemma 4.1]{Fuj03}.
    
    By the index conjecture in dimension 3--see, e.g., \cite{YXu}--
    there is a positive integer $I_0$ such that $I_0 \cdot (K_{X_\eta} + \Delta_\eta) \sim 0$.
    Then, the coefficients appearing in a terminalization of $({X_\eta} , \Delta_\eta)$ are determined by those of $({X_\eta} , \Delta_\eta)$.
    Thus, without loss of generality, we may replace $(X,\Delta)$ with a birational model such that $({X_\eta} , \Delta_\eta)$ is terminal.
    Furthermore, since the moduli part of the canonical bundle formula only depends on the generic fiber, by \cite{HX13} or \cite{Bir12}, we may choose a different compactification of $(X_\eta,\Delta_\eta)$ over $Y$ so that $(X,\Delta)$ is itself a terminal $\mathbb Q$-factorial pair and, after possibly replacing $Y$ with a birational model, with $K_X + \Delta \sim_{\mathbb Q,Y} 0$.

    By \cite[Theorem C]{EFGMS} and \cite[Theorem 1.4]{babwild}, the effective b-semiampleness conjecture \cite[Conjecture 7.13.3]{PS09} is known in relative dimension 2.
    Therefore, by \cite[Theorem 7.6]{Fil20} and its proof, the conjecture in relative dimension 3 reduces to two cases:
    $X \to Y$ is a $K_X$-Mori fiber space, or $X$ is terminal--this condition is not explicitly stated in {\it loc. cit.} but follows by applying the reduction at the beginning of this proof--and $K_X\sim_{\mathbb Q,Y}0$.
    More precisely, by the proof of \cite[Theorem 7.6]{Fil20}, any other fibration admits a factorization over the generic point of the base and in that case the validity of the conjecture is settled.
    So, only these two extreme cases remain open.
    
    In the case of a Mori fiber space, the claim follows by \cite[Corollary 7.8]{BFMT}.
    Thus, under our assumptions and reductions, we may assume that the generic fiber is terminal and has positive irregularity.
    But then, by \cite{Wit08}, the generic fiber admits an Albanese morphism $X_\eta \to Z_\eta$ to a variety $Z_\eta$ defined over the function field of $Y$ such that $Z_\eta$ is a torsor under an Abelian variety and, up to passing to the algebraic closure, $X_\eta \to Z_\eta$ provides the Albanese morphism of the geometric generic fiber.
    In turn, by \cite{Xu20}, $X_\eta \to Z_\eta$ is a fibration with connected fibers.
    We observe that \cite{Wit08} discusses the case of smooth varieties;
    since $X_\eta$ is terminal, so in particular klt, by \cite{HM07} the Albanese morphism of $X_\eta$ agrees with the one of any resolution.
    In particular, we reduced to the case where the fibration $X \to Y$ factors, at least rationally, as a chain of fibrations $X \dashrightarrow Z \dashrightarrow Y$, where $Z$ is a geometric realization of $Z_\eta$ over $Y$.
    In turn, the methods of \cite[\S~7]{Fil20} apply and the claim follows.    
\end{proof}

\subsection{Index results}

\autoref{intro_thm_index} is a special case of the following more general statement.

\begin{theorem}\label{index_thm_pairs}
    Let $\Gamma \subset [0,1]$ be a DCC set of rational numbers.
    Then, there exists a positive integer $C > 0$ such that the following holds.
    Let $(X,\Delta)$ be a log Calabi--Yau 4-fold with $\mathrm{coeff}(\Delta) \subset \Gamma$.
    Further assume that $X$ admits a fibration $f \colon X \to Y$ with $\dim Y > 0$.
    Then, $C \cdot (K_X + \Delta) \sim 0$ holds.
\end{theorem}

\begin{proof}
    By \cite[Theorem~1.4]{HMX14}, there is a finite subset $\Gamma_0 \subset \Gamma$ such that $\mathrm{coeff}(\Delta) \subset \Gamma_0$.
    By \cite[Corollary~1.7]{JL21}, we may assume that $(X,\Delta)$ is klt.
    By \cite[Theorem~3.1.5]{YXu}, we may assume $\Delta = 0$. 
    Lastly, by \cite[Corollary 1.4]{Jiao24},
    we may assume that $X$ has canonical singularities.
    
    Let $f \colon X \to Y$ be a 4-fold as in the statement.
    Let $(Y,B_Y,\bM.)$ be the generalized pair induced on $Y$ by the canonical bundle formula.
    In turn, we may replace $X$ with a $\QQ$-factorial terminalization.
    In particular, if $\dim X - \dim Y =2$, we may assume that the general fiber of $f$ is a smooth $K$-trivial surface.
    
    By \cite[Theorem 3.1]{FM00}, the b-Cartier index of $\bM.$ only depends on the middle Betti number and the index of a general fiber of $f$.
    If $\dim X -\dim Y =1$, this is clear.
    If $\dim X -\dim Y = 2$, these invariants are bounded above by the classification of surfaces; see, e.g., \cite{Bea96}.
    Lastly, if $\dim X - \dim Y =3$, this is the content of \autoref{thm_cartier_index}.
    Thus, the b-Cartier index of $\bM.$ is bounded.
    
    Similarly, by \cite[\S~7.5]{PS09}, we have
    \begin{equation}\label{eq_index}
    I\cdot K_X \sim I \cdot f^*(K_Y +B_Y + \bM Y.),
    \end{equation}
    where $I$ denotes the index of the generic fiber of $f$.
    Since the index conjecture is known up to dimension 3--see, e.g., \cite{YXu}--$I$ belongs to a finite set.
    Thus, to bound the index of $K_X$, it suffices to bound the index of $K_Y+B_Y+\bM Y.$.

    By construction, the coefficients of $B_Y$ are of the form $1-e$, where $e$ is a log canonical threshold over a codimension 1 point in $Y$; see \cite[\S~7.2]{PS09}.
    Thus, the numbers $e$ belong to the set of log canonical thresholds of varieties of dimension at most 4.
    By \cite[Theorem 1.1]{HMX14}, the values $e$ belong to an ACC set.
    In turn, the coefficients of $B_Y$ belong to a DCC set of rational numbers.
    Then, since the index conjecture for klt pairs in known up to dimension 3, we may conclude by \cite[Lemma 2.30]{FFMP}.
\end{proof}

\subsection{Iitaka volume results}

Consider a projective pair $(X,\Delta)$.
Its {\it Kodaira dimension $\kappa(X,\Delta)$} is defined as the Iitaka dimension of the $\mathbb Q$-Cartier divisor $K_X+\Delta$.
If $\kappa=\kappa(X,\Delta) \geq 0$, the {\it Iitaka volume $\mathrm{vol}_\kappa(X,\Delta)$} is defined as follows:
$$
\limsup \frac{h^0(X,\mathcal{O}_X(\lfloor m(K_X+\Delta) \rfloor))}{m^{\kappa}/\kappa!}.
$$

The study of the arithmetic of the possible sets of volumes is an important topic in birational geometry;
see, e.g., \cite{HMX13,HMX18}.
To this end, consider a DCC set $\Gamma \subset [0,1]$ and non-negative integers $d$ and $\kappa$.
Following \cite{CHL25}, we define
$$
\mathfrak{P}^{\Gamma}_{\mathrm{lc}}(d,\kappa) = \{(X,\Delta)|(X,\Delta) \text{ \rm is a projective lc pair},\dim X = d, \kappa(X,\Delta)=\kappa,\mathrm{coeff(\Delta)\subset \Gamma} \}
$$
and
$$
\mathfrak{V}^\Gamma_{\textrm{lc}}(d,\kappa) = \{\mathrm{vol}_\kappa(X,\Delta)|(X,\Delta) \in \mathfrak{P}^{\Gamma}_{\mathrm{lc}}(d,\kappa) \}.
$$
If $d = \kappa$, $\mathfrak{V}^\Gamma_{\textrm{lc}}(d,\kappa)$ is a DCC set by \cite{Ale94,HMX18}.
More recently, there has been interest in studying the set $\mathfrak{V}^\Gamma_{\textrm{lc}}(d,\kappa)$ when $0 < \kappa < d$;
see \cite{Li24,Bir21}.
In particular, by \cite{CHL25}, $\mathfrak{V}^\Gamma_{\textrm{lc}}(d,\kappa)$ is a DCC set for $d \leq 3$.

\begin{theorem}\label{iitaka volume d-3}
Let $\Gamma \subset [0,1]$ be a DCC set of rational numbers,
and let $d$ and $\kappa$ be non-negative integers satisfying $\kappa \leq d$.
Then, $\mathfrak{V}^\Gamma_{\textrm{lc}}(d,\kappa)$ is a DCC set for all $\kappa \geq d-3$.
\end{theorem}

\begin{proof}
Without loss of generality, we may replace $\Gamma$ with $\Gamma \cup \{1\}$ and thus assume $1 \in \Gamma$.
When $\kappa(K_X+\Delta) \geq d-2$, the result follows from \cite[Theorem 1.3]{CHL25}. 

So, we may assume $\kappa(K_X+\Delta) = d-3$.
If $(X,\Delta)$ is klt, by \cite[Proposition~3.2]{CHL25}, $(X,\Delta)$ has a good minimal model $(X^\prime, \Delta^\prime)$.
Replacing $(X,\Delta)$ with $(X^\prime, \Delta^\prime)$, we may assume $K_X+\Delta$ is semiample, inducing the Iitaka fibration $f \colon X \to Y$. 
If $(X,\Delta)$ is log canonical but not klt, the proof of \cite[Proposition~3.2]{CHL25} can be modified as follows.
For the reader's convenience, we include a full proof following the blueprint of \cite[Proposition~3.2]{CHL25}.

Let $f \colon X \dashrightarrow Y$ a model of the Iitaka fibration of $(X,\Delta)$, where $Y$ is a smooth variety (notice that the target is denoted by $Z$ in {\it loc. cit.}).
Let $h \colon X_\infty \to X$ be a log resolution of $X$ that also resolves $f$.
On $X_\infty$, we choose $\Delta_\infty = h^{-1}_*\Delta + E_\infty$ as boundary, where $E_\infty$ denotes the reduces $h$-exceptional divisor.
By \cite[Propostion~2.17]{ACSS}, we may replace $X_\infty$ and $Y$ with higher birational models so that $X_\infty \to Y$ is equidimensional;
as in the previous step, the exceptional divisors are added with coefficient 1 to the strict transform of the boundary, in order to preserve the log canonical ring of the pair.
After this replacement, $(X_\infty,\Delta_\infty)$ may no longer be log smooth, but it is still $\mathbb Q$-factorial and log canonical.
Then, by \cite[Theorem~1.4]{HH20} and \cite{KMM94} (see also \cite[Theorem~2.3]{Sho96} for the $\mathbb R$-coefficients version, consistent with the setup of \cite{HH20}), $(X_\infty,\Delta_\infty)$ admits a relative good minimal model over $Y$.
We denote by $(X',\Delta')$ said relative good minimal model over $Y$ and by $Y'$ its relative ample model, which is a birational model of $Y$.
The contraction $f' \colon (X^\prime,\Delta') \to Y'$ has the following properties:
\begin{itemize}
    \item the section rings of $(X,\Delta)$ and $(X',\Delta')$ coincide;
    \item the fibration $f'$ is a model of the Iitaka fibration for both $K_X+\Delta$ and $K_{X'} + \Delta'$;
    \item $\mathrm{coeff}(\Delta') \subset \mathrm{coeff}(\Delta) \cup \{1\}$; and
    \item $(X',\Delta')$ is a log Calabi--Yau pair over $Y'$, inducing a generalized pair $(Y',B_{Y'},\mathbf M)$.
\end{itemize}
By abuse of notation, we replace $(X,\Delta)$ with said model and we replace $Y$ with $Y'$.
In particular, we write $(Y,B_Y,\mathbf M)$ for the above generalized pair.
Notice that, since $f$ is a model for the Iitaka fibration of $(X,\Delta)$, then $K_Y+B_Y+\mathbf M _Y$ is big.

By global ACC, see \cite[Theorem 1.5]{HMX14}, the horizontal multiplicities of $\Delta$ lie in a fixed finite subset $\Gamma_0$ of $\Gamma$. 
By \autoref{thm_cartier_index}, there is a constant $I$, only depending on $\Gamma_0$, such that $I \cdot \mathbf M$ is b-Cartier.

Then we have $\mathrm{vol}_{\kappa}(X,\Delta) = \mathrm{vol}(K_Y+B_Y+\mathbf M _Y)$, and by ACC for log canonical thresholds, see \cite[Theorem 1.1]{HMX14}, the coefficients of $B_Y$ lies in a fixed DCC set. 

Finally, by DCC of volumes for generalized pairs, see \cite[Theorem~1.3]{Bir21}, $\mathrm{vol}_{\kappa}(X,\Delta) = \mathrm{vol}(K_Y+B_Y+\mathbf M _Y)$ lies in a DCC set.
\end{proof}

\begin{corollary}\label{cor:iitaka_volume}
Let $\Gamma \subset [0,1]$ be a DCC set of rational numbers,
and let $d$ and $\kappa$ be non-negative integers satisfying $\kappa \leq d$.
Then, $\mathfrak{V}_{\mathrm{lc}}^\Gamma(d,\kappa)$ is a DCC set for all $0 \leq d \leq 4$.
\end{corollary}

\begin{proof}
When $\kappa(X,\Delta) = 0$, $\mathrm{vol}_{\kappa}(X,\Delta) = 1$ holds and there is nothing to prove. 
When $\kappa(X,\Delta) \geq 1$, the result follows from \autoref{iitaka volume d-3}. 
\end{proof}

\bibliographystyle{amsalpha}
\bibliography{refs}

\end{document}